\documentclass[11pt,notitlepage,a4paper]{article}

\usepackage{amssymb} 
\usepackage[intlimits]{amsmath}
\usepackage{amsfonts}
\usepackage{amsthm} 

\usepackage{mathptmx}

\usepackage{hyperref}
\hypersetup{colorlinks=true,linkcolor=RoyalPurple,citecolor=RoyalPurple}

\usepackage{cite}

\usepackage{graphicx}

\usepackage{geometry}

\usepackage{color}
\usepackage[dvipsnames]{xcolor}

\renewcommand{\a }{\alpha }

\newcommand{\D }{\Delta }

\newcommand{\e }{\varepsilon }

\newcommand{\G }{\Gamma}

\newcommand{\n }{\nabla }
\newcommand{\vp }{\varphi }
\renewcommand{\phi}{\varphi}

\newcommand{\s }{\sigma }

\renewcommand{\t }{\tau }

\renewcommand{\O }{\Omega }

\newcommand{\ov}{\overline}

\newcommand{\be}{\begin{equation}}
\newcommand{\ee}{\end{equation}}

\newcommand{\R}{\mathbb{R}}
\newcommand{\N}{\mathbb{N}}

\newcommand{\C}{\mathbb{C}}

\newcommand{\de}{\partial}

\newcommand{\M}{\mathcal{M}}

\newcommand{\ra}{{\rangle}}
\newcommand{\la}{{\langle}}

\newcommand{\cN }{\mathcal{N}}
\newcommand{\cZ }{\mathcal{Z}}

\newcommand{\cV }{\mathcal{V}}

\newcommand{\cO }{\mathcal{O}}

\newcommand{\V}{\text{\tiny $V$}}

\newtheorem{Theorem}{Theorem}[section]
\newtheorem{Lemma}[Theorem]{Lemma}
\newtheorem{Proposition}[Theorem]{Proposition}
\newtheorem{Corollary}[Theorem]{Corollary}
\newtheorem{Remark}[Theorem]{Remark}
\newtheorem{Definition}[Theorem]{Definition}

\def\R{{\mathbb R}}

\def\C{{\mathcal C}}

\def\div{{\rm div}}
\def\eps{{\varepsilon}}

\font\sc=cmcsc9  \textwidth=14truecm
\title{Critical domains for the first nonzero Neumann eigenvalue in Riemannian manifolds}
\author{Mouhamed Moustapha Fall   and   Tobias Weth }

\begin{document}
\date{}
\maketitle
\let\thefootnote\relax\footnotetext{\hspace{-0.6cm} M. M. Fall (mouhamed.m.fall@aims-senegal.org): African Institute for Mathematical Sciences in Senegal (AIMS Senegal), 
KM 2, Route de Joal, B.P. 14 18. Mbour, S\'en\'egal\\ \noindent
T. Weth (weth@math.uni-frankfurt): Goethe-Universit\"{a}t Frankfurt, Institut f\"{u}r Mathematik.
Robert-Mayer-Str. 10 D-60054 Frankfurt, Germany.}
 
 \begin{flushright}
   \small{\textit{to the memory of Ahmed El Soufi.}}
 \end{flushright}

\begin{abstract}
The present paper is devoted to geometric optimization problems related to the Neumann eigenvalue problem for the Laplace-Beltrami operator on bounded subdomains $\O$ of a Riemannian manifold $(\M,g)$. More precisely, we analyze locally extremal domains for the first nontrivial eigenvalue $\mu_2(\Omega)$ with respect to volume preserving domain perturbations, and we show that corresponding notions of criticality arise in the form of overdetermined boundary problems. Our results rely on an extension of Zanger's shape derivative formula which covers the case when $\mu_2(\Omega)$ is not a simple eigenvalue. In the second part of the paper, we focus on product manifolds of the form $\M = \R^k \times \cN$, and we classify the subdomains where an associated overdetermined boundary value problem has a solution. 
\end{abstract}

\section{Introduction}
\label{sec:introduction}
Let  $(\M,g)$ be a complete Riemannian manifold of dimension
$N$, $N\geq2$. For a bounded smooth domain $\O\subset \M$ with $C^2$-boundary, we consider the  Neumann eigenvalue problem
\begin{equation}
  \label{eq:16}
-\D_g u = \mu\,u \qquad \text{in $\O$},\qquad \de_\eta u =0\quad
\text{on $\de\O$},
\end{equation}
where $\D_g u=div_g(\n u)$ is the Laplace-Beltrami operator of $u$ on $\M$, $\eta$ is the outer unit normal to $\de\O$ and $ \de_\eta u:= \la \n u,\eta\ra_g$. The set of eigenvalues, counted with multiplicities, in the above
eigenvalue problem is given as an increasing sequence $0=\mu_1(\O)<\mu_2(\O)\leq\dots$. Of particular interest is the first nontrivial eigenvalue $\mu_2(\Omega)$, characterized variationally as 
\begin{equation}
  \label{eq:characterization-mu2}
\mu_2(\O)=\inf \Bigl\{\frac{\int_{\O}|\n u|^2\,dx}{\int_{\O} u^2\,dx} \::\: u \in  H^1(\O) \setminus \{0\},\: \int_{\O}u\,dx =0\Bigl\}. 
\end{equation}
Here $H^1(\O)$ is the usual first order Sobolev space. A natural question is to study extremal values of $\mu_2(\Omega)$ among domains $\Omega \subset \M$ satisfying a volume constraint. 
By classical results of Szeg\"{o} and Weinberger (see \cite{Sz,Wein}),
balls maximize $\mu_2$ among domains $\Omega$ having fixed volume $|\Omega|= v>0$ in
$\M=\R^{N}$. We note that, if $B^N \subset \R^N$ is the unit ball, the eigenvalue $\mu_2(B^N)$ has multiplicity $N$ with corresponding eigenfunctions of the form   
$x\mapsto \vp(|x|)\frac{x_i}{|x|}$, $i=1,\dots,N$, see Section~\ref{s:cylinder-case} below for details.
 As remarked in \cite{chavel} and \cite{Ashb-Bengu}, the maximization property of balls extends to the case of the $N$-dimensional hyperbolic space. Moreover, the
same property is valid in a hemisphere
\cite{Ashb-Bengu} and -- under further restrictions on the domain -- also 
in rank-1 symmetric spaces \cite{AS}. On the other hand, the problem of globally minimizing $\mu_2$ among domains $\Omega$ having fixed volume $|\Omega|=v<|\M|$ has no solution, since $\mu_2(\Omega)$ approaches zero 
within the class of domains $\Omega$ built by connecting two disjoint subdomains with a thin tube. 

The present paper consists of two parts. In the first part, we characterize -- by means of overdetermined boundary value problems -- subdomains of a general Riemannian manifold which are {\em locally} maximizing or minimizing $\mu_2$ with respect to volume preserving domain variations. In the second part we focus on the special case of cylindrical manifolds of the form 
$\R^k \times \cN$, where more information can be derived. Here $\cN$ is a given closed Riemannian manifold. In this case we wish to determine global constrained maximizers for $\mu_2$ and classify solutions of an associated overdetermined boundary value problem.

To state our main results, we need to introduce some notation. Since we assume that $\M$ is complete, we have a globally defined exponential map $\exp_x : T_x \M \to \M$ at every $x \in \M$, and every bounded subset of $\M$ is relatively compact. For a nonnegative integer $k$, we let $\cO^k(\M)$ denote the class of all bounded subdomains $\Omega \subset \M$ with $C^k$-boundary. Moreover, we let $\cV^k(\M)$ denote the space of all $C^k$-vector fields on $\M$ with bounded covariant derivatives of order $i \le k$, which is a Banach space with canonical norm $\|\cdot\|_{C^k}$, see e.g. \cite{Aubin}. For $V \in \cV^k(\M)$, 
we define the map 
$$
\tau_\V \in \C^k(\M,\M),\qquad \tau_\V(x)= \textrm{Exp}_{x}(V(x)),
$$
and we put $\Omega_{\V}:= \tau_{\V}(\Omega)$ for $\O  \subset \M$.

\begin{Definition}
\label{admissible-deformations}{\rm 
Let $\Omega \subset \M$ be a bounded domain. We say that $V \in \cV^1(\M)$ is an {\em admissible deformation field for $\Omega$} if $\tau_V$ maps a neighborhood of $\overline{\Omega}$ diffeomorphically onto a neighborhood of $\overline{\Omega_V}$ and $|\Omega_V|= |\Omega|$.} 
\end{Definition}

The requirement $V \in \cV^1(\M)$ in this definition guarantees  -- in particular -- that $\Omega_V$ has a $C^1$-boundary if this is true for $\Omega$. We can now define the notion of constrained local extrema for $\mu_2$. 

\begin{Definition}
\label{sec:definition-local-extrema}{\rm 
Let $\Omega  \in \cO^1(\M)$. We say that $\Omega$ is a {\em constrained local maximum for $\mu_2$} if there exists $\eps>0$ such that for every admissible deformation field $V \in \cV^1(\M)$ for $\Omega$ with $\|V\|_{C^1}<\eps$ we have $\mu_2(\Omega_\V)\le \mu_2(\Omega)$. 
If this inequality is strict in the case where $\Omega_\V \not = \Omega$, we call $\Omega$ a {\em strict constrained local maximum}. 
Constrained local minima are defined in an analogous way via the opposite inequalities.}
\end{Definition}

Finally, we define corresponding notions of criticality. The main difficulty here is the fact that $\mu_2(\Omega)$ may or may not be a simple eigenvalue. In the case where $\mu_2(\Omega)$ is simple, Zanger's formula \cite{Zanger} for the shape derivative of Neumann eigenvalues with respect to domain variations gives rise to a straightforward notion of criticality which we will refer to as {\em criticality in strong sense} in the following, see Definition~\ref{critical-points} below.
In the case where $\mu_2(\Omega)$ is degenerate, $\mu_2$ in general does not have shape derivatives at $\Omega$ and thus
Zanger's formula is not valid. In Proposition~\ref{zanger-formula} below we will derive a useful variant for one-sided shape derivatives. In contrast to the argument by Zanger in \cite{Zanger}, our derivation solely relies on the variational characterization of $\mu_2(\Omega)$, and we believe that the resulting formula does not extend to higher Neumann eigenvalues. On the other hand, the formula allows to conclude that constrained local minima for $\mu_2$ are critical in strong sense, whereas in the case of constrained local maxima it gives rise to a weaker notion of criticality. The precise notions of weak and strong criticality used in this paper are the following.  

\begin{Definition}
\label{critical-points}{\rm 
Let $\Omega \in \cO^1(\M)$.
\begin{enumerate}
\item[(i)] We say that $\Omega$ is a {\em constrained critical point for $\mu_2$ in strong sense} if, for some constant $\lambda \in \R$, there exists a solution $u \not = 0$ of the overdetermined problem 
\begin{equation}
\label{eq:overdetermined-0}
\left \{
\begin{aligned}
-\D_g u&=\mu_2(\Omega) u &&\qquad \text{in $\O$},\\
\partial_\eta u&=0,\quad |\nabla u|^2 -\mu_2(\Omega) u^2 = \lambda    
&&\qquad \text{on $\de \O$.}
  \end{aligned}
\right.
\end{equation}

\item[(ii)] We say that $\Omega$ is a {\em constrained critical point for $\mu_2$ in weak sense} if there exists finite many solutions $u_1,\dots,u_m \in C^2(\Omega)\setminus \{0\}$ of the Neumann eigenvalue problem 
\begin{equation*}
\left \{
\begin{aligned}
-\D_g u_i&=\mu_2(\Omega) u_i &&\qquad \text{in $\O$},\\
\partial_\eta u_i&=0&&\qquad
\text{on $\de \O$},\\
  \end{aligned}
\right.
\end{equation*}
with the property that $\sum \limits_{i=1}^m \Bigl(|\nabla u_i|^2 -\mu_2(\Omega) u_i^2\Bigr) = \lambda$ on $\de \O$ for some constant $\lambda \in \R$.
\end{enumerate}}
\end{Definition}

The weak notion of criticality defined here is inspired by \cite{Soufi-Ilias,Nadi}. The first main result of the present paper is the following. 

\begin{Theorem}
\label{sec:main-theorem-1}
Let $\Omega \in \cO^2(\M)$.\\
(i) If $\Omega$ is a constrained local maximum for $\mu_2$, then $\Omega$ is a constrained critical point for $\mu_2$ in weak sense.\\
(ii) If $\Omega$ is a constrained critical point for $\mu_2$ in strong sense and $\partial \Omega$ is connected, then it is a strict constrained local maximum for $\mu_2$.\\ 
(iii) If $\Omega$ is a constrained local minimum with respect to domain variations, then $\M$ is compact and  $\Omega = \M$.
\end{Theorem}

Some remarks are in order. It is already evident from the euclidean case $\M = \R^N$ that, in general, criticality in weak sense cannot be improved to criticality in strong sense for constrained local maxima $\Omega \in \cO^2(\M)$ for $\mu_2$. Indeed, by Weinberger's result discussed above, the unit ball $\Omega = B^N \subset \R^N$ is a constrained global (and thus local) maximizer, and it does not admit a solution of the overdetermined problem (\ref{eq:overdetermined-0}) unless $N=1$. On the other hand, we shall see in Corollary~\ref{overdetermined-corollary} below that constrained local minima $\Omega \in \cO^2(\M)$ for $\mu_2$ are critical in strong sense, and from this we will deduce Theorem~\ref{sec:main-theorem-1}(iii). 

As indicated already, the proof of Theorem~\ref{sec:main-theorem-1} relies on the calculation of one-sided shape derivatives along curves of admissible deformation fields for $\Omega$. In the case where $\Omega \in \cO^2(\M)$, these curves are closely related to $C^1$-functions $h: \partial \Omega \to \R$ with $\int_{\partial \Omega} h\,d\sigma= 0$. More precisely, for any such function $h$, there exists $\eps_0>0$ and a $C^1$-curve $(-\eps_0,\eps_0) \to \cV^1(\M)$, $t \mapsto V_t$ of admissible deformation fields for $\Omega$
with $V_0 = 0$ and $\partial_t \big|_{t=0} V_t \equiv h \eta$ on $\partial \Omega$, where $\eta$ is the outer unit normal on $\partial \Omega$ as before. This fact is rather well known (at least in the euclidean case, see e.g. \cite{henrot-book-1,henrot-book-2}), and we give a short proof for the convenience of the reader in Lemma~\ref{sec:technical-lemma-1}(i) below. Note that we require $\Omega \in \cO^2$ to guarantee that $\eta$ and therefore $V$ are of class $C^1$.

In \cite{Soufi-Ilias}, the authors derive a similar notion of criticality in weak sense for locally extremals of higher Dirichlet eigenvalues on $-\Delta_g$ on $\M$ with respect to variations of the metric $g$. With regard to the underlying methods, the present paper differs from \cite{Soufi-Ilias}
as we use the variational characterization of $\mu_2(\Omega)$ instead of Kato's analytic perturbation theory used in \cite{Soufi-Ilias}.

We shall see in Remark~\ref{on-the-connectedness} that the connectedness assumption on $\partial \Omega$ in Theorem~\ref{sec:main-theorem-1}(ii) cannot be removed. On the other hand, a more general version of Theorem~\ref{sec:main-theorem-1}(ii) -- not requiring the connectedness of $\partial \Omega$ -- is available when the class of admissible deformation fields is reduced. For $\eps>0$ and a compact  subset $K$ of $\M$, we denote by $U_\eps(K)$ the $\e$-tubular neighborhood of $K$ in $\M$. 

\begin{Definition}
\label{def-locally-volume-preserving}
Let $\eps>0$. We say that an admissible deformation field $V \in \cV^1(\M)$ for $\Omega \in \cO^1(\M)$ is {\em locally volume preserving in $U_\eps(\partial \Omega)$ } if $(\Omega_\V \setminus \Omega) \cup (\Omega \setminus \Omega_\V) \subset U_\eps(\partial \Omega)$ and 
$$
|(\Omega_\V \setminus \Omega) \cap A| = |(\Omega \setminus \Omega_\V) \cap A|\qquad \text{for every connected component $A$ of $U_\eps(\partial \Omega)$.}
$$
\end{Definition}

In the case where $\Omega \in \cO^2(\M)$, locally volume preserving admissible deformation fields for $\Omega$ can, similarly as remarked above, be constructed starting from $C^1$-functions $h: \partial \Omega \to \R$ with the property that
$\int_{\Gamma} h\,d\sigma = 0$ for every connected component  $\G \subset \partial \Omega$, see Lemma \ref{sec:technical-lemma-1}(ii) below.
We then have the following generalization of Theorem~\ref{sec:main-theorem-1}(ii).

\begin{Theorem}
\label{sec:special-theorem-1}
If $\Omega \in \cO^2(\M)$ is a constrained critical point for $\mu_2$ in strong sense, then 
there exists $\eps>0$ such that $\mu_2(\Omega_\V) < \mu_2(\Omega)$ for every 
admissible deformation field $V$ for $\Omega$ which is locally volume preserving in $U_\eps(\partial \Omega)$ and such that $\Omega_\V \not = \Omega$.
\end{Theorem}

Next we restrict our attention to cylindrical manifolds of the type $\M:=\R^k \times \cN$, where $(\cN,g_\cN)$ is a closed connected manifold and the product metric $g= g_{eucl} \otimes g_{\cN}$ is considered on $\M$. For the problem of maximizing $\mu_2(\Omega)$ among domains of fixed volume $v$, one may expect a different shape of maximizers depending on the size of $v$. If $v>0$ is small, the results in \cite{fall-weth-sw-profile} on the corresponding asymptotic profile expansion suggest that maximizing domains are perturbations of small geodesic ellipsoids in $\M$, whereas for large $v$ the domains  
$$
\Omega_r:= \{(t,x) \in \M \::\; t \in \R^k,\,|t| \le r,\, x \in \cN\} \:\subset \:\M, \qquad r>0
$$
are natural candidates for maximizers in view of Weinberger's result \cite{Wein} for the euclidean case. The following result partially supports this intuition. For this we consider the (critical) volume parameter 
$$
v_c:= \Bigl(\frac{\mu_2(B^k)}{\mu_2(\cN)}\Bigr)^{\frac{k}{2}} \omega_k |\cN|. 
$$
Here $B^k \subset \R^k$ denotes the unit ball with volume $\omega_k$, $|\cN|$ denotes the volume of $\cN$, and $\mu_2(\cN)$ resp. $\mu_2(B^k)$ denote the first nontrivial Neumann eigenvalues of 
$-\Delta_{g_{\text{\tiny $\cN$}}}$, $-\Delta_{g_{\textrm{\tiny eucl}}}$ on $\cN$, $B^k$, respectively.

\begin{Theorem}
\label{sec:weinberger-modified}
Let $v  \ge v_c$ and $r= \Bigl(\frac{v}{\omega_k|\cN|}\Bigr)^{\frac{1}{k}}$, so that $|\Omega_r| = v$. Then we have 
$$
\mu_2(\Omega) \le \mu_2(\Omega_r) \qquad \text{for every domain $\Omega \in \cO^1(\M)$ with $|\Omega|=v$.}
$$
Moreover, equality holds if and only if $\Omega$ coincides with $\Omega_r$ up to translation in the $\R^k$-variable.
\end{Theorem}

It remains open whether the $v_c$ is optimal in Theorem~\ref{sec:weinberger-modified}. The value $v_c$ is critical in the sense there exists eigenfunctions for $\mu_2(\Omega_r)$ which do not depend on the $\cN$-variable if and only if $v \ge v_c$, i.e., if and only if $r \ge \Bigl(\frac{\mu_2(B^k)}{\mu_2(\cN)}\Bigr)^{\frac12}$. This property is essential for the proof of Theorem~\ref{sec:weinberger-modified}, which is modeled on Weinberger's argument in \cite{Wein}. 

In the case where $k=1$, the domains $\Omega_r \subset \M$, $r \ge \Bigl(\frac{\mu_2(B^k)}{\mu_2(\cN)}\Bigr)^{\frac12}$ also have the special property of being constrained critical points for $\mu_2$ in strong sense. The following results shows that, up to translation, these are the only examples arising in this setting.

\begin{Theorem}
\label{sec:introduction-3-intro}
Let $\Omega \in \cO^2(\M)$ be a domain such that the overdetermined problem 
\begin{equation*}
\left \{
\begin{aligned}
-\D_g u&=\mu_2(\Omega) u &&\qquad \text{in $\O$},\\
\partial_\eta u&=0,\quad
|\nabla u|^2 -\mu_2(\Omega) u^2 = \lambda    
&&\qquad \text{on $\de \O$}.
  \end{aligned}
\right.
\end{equation*}
admits a solution for some constant $\lambda \in \R$. Then $k=1$ and $\Omega=\Omega_r$ for some $r \ge \Bigl(\frac{\mu_2(B^k)}{\mu_2(\cN)}\Bigr)^{\frac{1}{2}}$ up to translation in the $t$-variable.
\end{Theorem}

The proof of Theorem~\ref{sec:introduction-3-intro} is not straightforward, as it combines the analysis of partial derivatives of eigenfunctions (with respect to the $t$-variable) with estimates on the number of nodal domains and a sliding argument using the cylindrical structure of the problem. We recall that, in the euclidean setting, the sliding method has been developed in \cite{berestycki-nirenberg:91}.

\begin{Remark}
\label{on-the-connectedness}{\rm 
In the case $k=1$, $r \ge \Bigl(\frac{\mu_2(B^k)}{\mu_2(\cN)}\Bigr)^{\frac{1}{2}}$ the domain $\Omega=\Omega_r \subset \M$ is a constrained global maximizer for $\mu_2$ and a constrained critical points for $\mu_2$ in strong sense, but it is not a {\em strict} constrained local maximizer. Indeed, for given $\eps>0$, one may consider an admissible deformation field $V \in \cV^1(M)$ for $\Omega$ with $\|v\|_{C^1}< \eps$ and such that $\Omega_V$ is a mere translation of $\Omega$ in the $t$-variable, which implies that $\mu_2(\Omega_V)= \mu_2(\Omega)$.
This shows that we cannot remove the additional assumptions on $\partial \Omega$ or $V$ in Theorems \ref{sec:main-theorem-1}(ii) and \ref{sec:special-theorem-1}.}
\end{Remark}

The paper is organized as follows. Section~\ref{sec:preliminaries} contains two preliminary lemmas. The first provides an expansion of metrics associated with  domain deformations, and the second ensures the existence of suitable curves of  admissible deformation fields. In Section~\ref{zanger-formula}, we prove a one-sided variant of Zanger's shape derivative formula which holds without requiring simplicity of $\mu_2$. From this formula, we then derive the solvability of associated overdetermined boundary value problems, and by this we complete the proof of Theorem~\ref{sec:main-theorem-1}. In Section \ref{s:cylinder-case} we restrict our attention to the case of cylindrical manifolds, and we prove Theorems
\ref{sec:weinberger-modified} and \ref{sec:introduction-3-intro}.\\

\textbf{Acknowledgement:}
M.M.F. is supported by the Alexander von Humboldt Foundation. Moreover, M.M.F. and T.W. wish to thank the German Academic Exchange Service (DAAD) for funding of this work within the program 57385104 ``Local and nonlocal effects in geometric variational problems''.

\section{Preliminaries}
\label{sec:preliminaries}
In this section we state and prove two preliminary lemmas. We start with a lemma on the expansion of a pullback metric under a curve of diffeomorphisms generated by a corresponding curve of vector fields. 

\begin{Lemma}
\label{lemma-appendix}
Consider a $C^1$-curve $(-\eps_0,\eps_0) \to \cV^1(\M)$, $\eps \mapsto V_\eps$  of vector fields $V_\eps \in \cV^1(\M)$ with $V_0 = 0$ 
and 
the maps 
$$
\tau_\eps \in \C^1(\M,\M), \qquad \tau_\eps(x)= \textrm{Exp}_{x}(V_\eps(x)), \qquad \eps \in (-\eps_0,\eps_0).
$$
Moreover, let $g_\eps$ denote the pull back of the metric $g$ under the map $\tau_\eps$ for $\eps \in (-\eps_0,\eps_0).$
In local coordinates $x_1,\dots,x_N$, setting $\de_i= \frac{\de}{\de x_i}$ and $g_{ij}=\la\de_i, \de_j \ra_g$, $g_{\e\!,\,ij}=\la\de_i, \de_j \ra_{g_\e}$, we then have the locally uniform expansions 
\begin{align}
g_{\e\!,\,ij}&= g_{ij} + \eps \langle \nabla_{\partial_j} V, \partial_i \rangle_g + \eps \langle \nabla_{\partial_i} V, \partial_j \rangle_g +o(\eps), \label{expansion-1-app}\\
g_\eps^{\;\; ij}&= g^{ij} - \eps \langle \nabla_{\partial_j} V, \partial_i \rangle_g - \eps \langle \nabla_{\partial_i} V, \partial_j \rangle_g +o(\eps),\label{expansion-2-app}
\end{align}
as $\eps \to 0$, where $V:= \partial_\eps \big|_{\eps=0} V_\eps$ and $(g_\eps^{\;\; ij})_{ij}$ denotes the inverse of $(g_{\e\!,\,ij})_{ij}$. Moreover, for the volume form of $g_\eps$ we have the expansion 
\begin{equation}
dv_{g_\eps}(x)= \Bigl(1+ \eps \, \div_g V + o(\eps)\Bigr)dv_g(x) \quad \text{as $\eps \to 0$.}
  \label{eq:expansion-3-app}
\end{equation}
\end{Lemma}

\begin{proof}
We first prove (\ref{expansion-1-app}). Fix $x\in \M$ and $w\in T_x\M$. Then we have 
$$
D\t_\e(x)[w]=\frac{d}{dt}\Big |_{t=0}\t_\e( \a(t)), 
$$
where $\a:(-\e,\e)\to \M$ is a smooth curve with $\a(0)=x$ and $\frac{d \a}{dt}(0)=w$.  
Note that the curve $\e\mapsto \t_\e(x)$ satisfies $\t_0(x)=x$ and 
$$
\frac{d}{d\e}\Big|_{\e=0} \t_\e(x)=  \frac{d}{d\e} \Big|_{\e=0} \textrm{Exp}_{x}(V_\eps(x)) =
\frac{d}{d\e} \Big|_{\e=0} V_\eps(x)= D \textrm{Exp}_{x}(0)V(x) = V(x).
$$
Therefore, denoting by $\frac{D}{d}$ covariant  derivatives along curves, we get 
\begin{align}
\frac{D}{d \e}\Big|_{\e=0}D\t_\e(x)[w]&= \frac{D}{d \e}\Big|_{\e=0} \frac{d }{dt}\Big|_{t=0}  \t_\e( \a(t))= \frac{D}{d t}\Big|_{t=0}\frac{d}{d\e}\Big|_{\e=0} \t_\e( \a(t))  \nonumber\\
&= \frac{D}{d t}\Big|_{t=0} V(\a(t))=\n_{ \frac{d \a}{dt}}V(\a(t)) \Big|_{t=0} =[\n_w V](x). \label{eq:mix-deriv}
\end{align}
Next, we consider local coordinates $(y_1,\dots,y_N)$ in a neighborhood of $x$. Moreover, we write $\de_i= \frac{\de}{\de y_i}$ for the corresponding coordinate vector fields and 
$$
g_{ij}=\la\de_i, \de_j \ra_g, \qquad \text{as well as}\qquad  g_{\e\!,\,ij}=\la\de_i, \de_j \ra_{g_\e} \quad \text{for $\eps \in (-\eps_0,\eps_0)$.}
$$
For fixed $i,j$, the function  
\begin{equation}
\label{x-eps-c-1}
(x,\eps) \mapsto g_{\e\!,\,ij}(x)=\la D\t_\e(x)[ \de_i],  D\t_\e(x)[ \de_j] \ra_g\Big|_{\t_\e(x)}
\end{equation}
then satisfies $g_{0\!,\,ij}(x)=g_{ij}(x)$ and, by \eqref{eq:mix-deriv}, 
\begin{align*}
\partial_{\eps} \Big|_{\eps=0} g_{\e\!,\,ij}(x)&=\Bigl \langle \frac{D}{d \e} \Big|_{\e=0}  D\t_\e(x)[ \de_i],  D\t_0(x)[ \de_j] \Bigr \rangle_g
 + \Bigl\langle D\t_0(x)[ \de_i],  \frac{D}{d \e}\Big |_{\e=0}  D\t_\eps(x)[ \de_j] \Bigr \rangle_g\\
&= \Bigl( \langle  \n_{\de_i} V,\de_j  \rangle_g+ \langle  \n_{\de_j} V,\de_i  \rangle_g \Bigr) \Big|_{x}.
\end{align*}
Consequently, we have that 
$$
g_{\e\!,\,ij}(x)= g_{ij}(x) + \eps \Bigl( \langle  \n_{\de_i} V,\de_j  \rangle_g+ \langle  \n_{\de_j} V,\de_i  \rangle_g \Bigr) \Big|_{x}
 +o(\eps)
$$
as $\eps \to 0$. Moreover, this expansion is locally uniform in $x$, since it follows from the assumption that 
the functions $(x,\eps) \mapsto \partial_\eps \, g_{\e\!,\,ij}(x)$, $i,j,=1,\dots,N$ are continuous in $x$ and $\eps$. Hence (\ref{expansion-1-app}) holds, and (\ref{expansion-2-app}) is a direct consequence of (\ref{expansion-1-app}). It thus remains to derive (\ref{eq:expansion-3-app}) from (\ref{expansion-1-app}). For this we note that 
\begin{equation}
  \label{eq:rel-representation}
dv_{g_\eps}(x) = \sqrt{\frac{|g_\eps|}{|g|}}dv_{g}(x) \quad \text{with $|g|:= \det (g_{ij})$, $|g_\eps|:= \det (g_{\e\!,\,ij})$.}
\end{equation}
Moreover, writing $V= V^k \de_k$ in local coordinates, we see that   
\begin{align*}
|g_\eps| &= |g|\Bigl(1 + \eps g^{ij} \partial_{\eps}\Big|_{\eps=0} g_{\e\!,\,ij} + o(\eps)\Bigr)= |g|\Bigl(1 + \eps g^{ij} \Bigl( \la  \n_{\de_i} V,\de_j  \ra_g+ \la  \n_{\de_j} V,\de_i  \ra_g\Bigr) + o(\eps)\Bigr)\\
&= |g|\Bigl(1 + \eps  g^{ij} \Bigl[\Bigl(\de_i V^k + V^\ell \Gamma^{k}_{\ell i}\Bigr) \la \de_k,\de_j  \ra_g
+ \Bigl(\de_j V^k + V^\ell \Gamma^{k}_{\ell j}\Bigr) \la \de_k,\de_i  \ra_g \Bigr)\Bigr] + o(\eps) \Bigr)\\
&= |g|\Bigl(1 + \eps  g^{ij} \Bigl[\Bigl(\de_i V^k + V^\ell \Gamma^{k}_{\ell i}\Bigr) g_{k j}
+ \Bigl(\de_j V^k + V^\ell \Gamma^{k}_{\ell j}\Bigr) g_{k i} \Bigr)\Bigr]+ o(\eps) \Bigr)\\
&= |g|\Bigl(1 + 2 \eps   \Bigl(\de_k V^k + V^\ell \Gamma^{k}_{\ell k} \Bigr)+o(\eps) \Bigr)= |g|\Bigl(1 + 2 \eps \, \div_g V+o(\eps) \Bigr)
\end{align*}
and consequently $\sqrt{\frac{|g_\eps|}{|g|}} = 1 + 2 \eps \, \div_g V+o(\eps)$ as $\eps \to 0$.
Combining this expansion with (\ref{eq:rel-representation}), we obtain (\ref{eq:expansion-3-app}). Moreover, the expansion is locally uniform since this is the case for (\ref{expansion-1-app}). This ends the proof. 
\end{proof}

The next lemma ensures the existence of curves of admissible deformation fields for a given domain $\Omega \in \cO^2(\M)$. It follows in a straightforward way 
from a well known rate of change formula for the volume functional (see e.g. \cite{henrot-book-1,henrot-book-2}), but we prefer to give a proof for the convenience of the reader. 

\begin{Lemma}
\label{sec:technical-lemma-1}
Let $\Omega \in \cO^2(\M)$, let $h: \partial \Omega \to \R$ be a $C^1$-function and $\eps>0$. Then there exists $\eps_0>0$ and a $C^1$-curve $(-\eps_0,\eps_0) \to \cV^1(\M)$, $t \mapsto V_t$ 
with 
$$
V_0 = 0,\qquad \partial_t \Bigl|_{t=0} V_t \equiv h \eta  \quad \text{on $\partial \Omega$} 
$$
and the following properties:
\begin{enumerate}
\item[(i)] If $\int_{\partial \Omega} h\,d\sigma = 0$, then $V_t$ is an admissible deformation field for $\Omega$ for $t \in (-\eps_0,\eps_0)$.
\item[(ii)] If $\int_{\Gamma} h\,d\sigma = 0$ for every connected component  $\G \subset \partial \Omega$, then, for $t \in (-\eps_0,\eps_0)$, $V_t$ is an admissible deformation field for $\Omega$ which is locally volume preserving in $U_\eps(\partial \Omega)$.
\end{enumerate}
Moreover, if $h \not \equiv 0$, then $\Omega_{\V_t}  \not = \Omega$ for every $t \in (-\eps_0,\eps_0)$.
\end{Lemma}

\begin{proof}
Let $W \in \cV^1(\M)$ be an arbitrary extension of the outer normal $\eta$ on $\partial \Omega$, and let $\tilde h \in C^1(\M)$ be an extension of $h$ to $\M$.  We first consider the case where $\int_{\partial \Omega} h\,d\sigma = 0$, as assumed in (i).  We then define the $C^1$-function 
$$
\R \times \R \to \R, \qquad (t,\delta) \mapsto  |\Omega_{\text{\tiny $V_{t,\delta}$}}| \quad \text{with $V_{t,\delta}= (t \tilde h + \delta)W \in \cV^1(\M)$.}  
$$
By the volume element expansion given in the appendix, Lemma \ref{lemma-appendix}, we then have that 
$$
\frac{\partial}{\partial \delta}\Bigl|_{(t,\delta)=(0,0)} |\Omega_{\text{\tiny $V_{t,\delta}$}}|  = \int_{\Omega} \div_g W d x= \int_{\partial \Omega} \langle W , \eta \rangle_g d \sigma = |\partial \Omega|>0.
$$
Hence the implicit function theorem yields the existence of $\eps_0>0$ and a $C^2$-function $(-\eps_0,\eps_0) \to \R$, $\:t \mapsto \delta(t)$ such that, setting $V_t:= (t \tilde h + \delta (t))W \in \cV(\M)$, we have $|\Omega_{\text{\tiny $V_t$}}|= |\Omega|$ for $t \in (-\eps_0,\eps_0)$ and thus, again by Lemma \ref{lemma-appendix}, 
$$
0 = \partial_t \Big |_{t=0}|\Omega_{\text{\tiny $V_t$}}|
= \int_{\Omega} \div_g [(\tilde h  +  \dot \delta(0))W] d x= \int_{\partial \Omega} (h  +  \dot \delta(0)) d \sigma = \dot \delta(0) |\partial \Omega|.
$$
We conclude that $\dot \delta(0)=0$. Hence $\partial_t \Bigl|_{t=0} V_t =  \tilde h W$, which coincides with $h \eta$ on $\partial \Omega$. Moreover, if 
$h \not \equiv 0$, we may make $\eps_0>0$ smaller if necessary to guarantell that $\Omega_{\V_t}  \not = \Omega$ for every $t \in (-\eps_0,\eps_0)$. Hence the claim holds.

We now consider the case where $\int_{\Gamma} h\,d\sigma = 0$ for every connected component  $\G\subset \partial \Omega$, as assumed in (ii).
Making $\eps>0$ smaller if necessary, we may assume, by the compactness of $\partial \Omega$, that the set 
$U_\eps(\partial \Omega)$ has finitely many connected components $A_1,\dots,A_n$. For $i=1,\dots,n$, let $\Gamma_i:= \partial \Omega \cap A_i$, and let $W_i \in \cV(\M)$ be a vector field supported in $A_i$ which coincides with the outer unit normal $\eta$ on $\Gamma_i$. 

Similarly as above, the implicit function theorem yields the existence of $\eps_0>0$ and a $C^1$-function 
$$
(-\eps_0,\eps_0) \to \R^n, \qquad t \mapsto \delta(t)= (\delta_1(t),\dots,\delta_n(t))
$$
such that, setting $V_t:= \sum \limits_{i=1}^{n} (t \tilde h + \delta_i(t))W_i \in \cV(\M)$, we have    
$$
|\Omega_{\text{\tiny $V_t$}} \cap A_i|= |\Omega \cap A_i|
\qquad \text{for $i=1\dots,n,\; t \in (-\eps_0,\eps_0)$}.
$$
Moreover, making $\eps_0$ smaller if necessary, we may assume that 
$$
(\Omega_{\text{\tiny $V_t$}} \setminus \Omega) \cup (\Omega \setminus \Omega_{\text{\tiny $V_t$}}) \subset U_\eps(\partial \Omega) \qquad \text{for $t \in (-\eps_0,\eps_0)$}.
$$
Lemma \ref{lemma-appendix} then implies that 
\begin{align*}
0 = \partial_\eps \Big |_{\eps=0}|\Omega_{\text{\tiny $V_t$}} \cap A_i|
&= \int_{\Omega \cap A_i} \! \div_g \Bigl(\sum_{j=1}^n [(\tilde h  +  \dot \delta_j(0))W_j]\Bigr) d x
= \int_{\Omega \cap A_i} \! \div_g \bigl[(\tilde h  +  \dot \delta_i(0))W_i\bigr] d x\\ 
&= \int_{\Gamma_i} \langle (h  +  \dot \delta_i(0))W_i , \eta \rangle_g d \sigma = \int_{\Gamma_i} h d\sigma + \dot \delta_i(0) |\Gamma_i|. 
\end{align*}
Since $\int_{\Gamma_i} h d\sigma= 0$ for $i=1,\dots,n$, we conclude that $\dot \delta_i(0)= 0$ for $i=1,\dots,n$. Consequently, we have
$\partial_t \Bigl|_{t=0} V_t = \sum \limits_{i=1}^{n}  \tilde h W_i$, and the RHS coincides with $h \eta$ on $\partial \Omega$. If $h \not \equiv 0$, we may again make $\eps_0>0$ smaller if necessary to guarantell that $\Omega_{\V_t}  \not = \Omega$ for every $t \in (-\eps_0,\eps_0)$. The claim follows.
\end{proof}

\section{A variant of Zanger's domain variation formula and its consequences}

In this section we extend Zanger's formula for the domain dependance of Neumann eigenvalues in the case of the variational eigenvalue $\mu_2$. Note that, in the case where $\mu_2(\Omega)$ is not a simple eigenvalue, $\mu_2$ is usually not a differentiable with respect to regular variations of $\Omega$. Nevertheless, the following one-sided derivative can be calculated. 

\begin{Proposition}
\label{zanger-formula} 
Let $\Omega \in \cO^2(\M)$, and let $(-\eps_0,\eps_0) \to \cV^1(\M)$, $\eps \mapsto V_\eps$ be a $C^1$-curve with $V_0=0$ and $V:= \partial_\eps\Big|_{\eps=0} V_\eps$. Then we have 
\begin{equation}
  \label{eq:zanger-generalized}
\partial_\eps^+ \Big|_{\eps=0} \mu_2(\Omega_{V_\eps}) = 
\min \Bigl \{\int_{\partial \Omega} (|\nabla u|^2- \mu_2(\Omega) u^2)\langle V,\eta  \rangle_g\,d\sigma\::\: 
u \in L\Bigr \},
\end{equation}
where $L \subset C^1(\overline \Omega)$ is the set of all Neumann eigenfunctions $u$ of $-\Delta_g$ on $\Omega$ corresponding to the eigenvalue $\mu_2(\Omega)$ with $\int_{\Omega}u^2\,dx = 1$.
\end{Proposition}

\begin{proof}
We start with some preliminary considerations. We first simplify the notation defined in the introduction, 
writing $\Omega_{\eps}$ in place of $\Omega_{\text{\tiny $V_\eps$}}$ and $\tau_\eps$ in place of $\tau_{\text{\tiny $V_\eps$}}$ for $\eps \in (-\eps_0,\eps_0)$. In the following, we let $g_\eps$ denote the pull back of the metric $g$ under the map $\tau_\eps$. Since $\tau_\eps : (M, g_\eps) \to (M,g)$ is an isometry and $\Omega_\eps= \tau_\eps(\Omega)$, the variational characterization for $\mu_2(\Omega_\eps)$ can be rewritten as 
\begin{equation}
  \label{eq:var-char-mu-2-rewritten}
\mu_2(\Omega_\eps) = \inf \Bigl\{\frac{\int_{\Omega} | \nabla_{g_\eps} u|^2_{g_\eps}\,dv_{g_\eps}}{\int_{\Omega}[u- m(u,\eps)]^2\,dv_{g_\eps}}\::\: u \in H^1(\Omega), u \not \equiv {\rm const} \Bigr\},
\end{equation} 
where 
$$
m(u,\eps)= \frac{1}{|\Omega|_{g_\eps}} \int_{\Omega} u \,dv_{g_\eps} \qquad \text{for $u \in H^1(\Omega)$}
$$
and $dv_{g_\eps}$ denotes the volume element with respect to the metric $g_\eps$. To prove the assertion, we thus need to use the expansions for the metric $g_\eps$ derived in Lemma~\ref{lemma-appendix}. For vector fields $w,z$ defined in $\overline \Omega$, locally written 
as $w = w^i \partial_i$, $z = z^j \partial_j$, (\ref{expansion-1-app}) gives rise to the expansion 
\begin{align}
\langle w, z\rangle_{g_\eps} & = g_{\e\!,\,ij}\: w^i z^j = \langle w, z\rangle_{g} +
\eps \langle \nabla_{\partial_j} V, \partial_i \rangle w^i z^j + \eps \langle \nabla_{\partial_i} V, \partial_j \rangle w^i z^j +o(\eps), \nonumber\\
&=\langle w, z\rangle_{g} +
\eps \langle \nabla_{Z} V, W  \rangle_g  + \eps \langle \nabla_{W} V, Z \rangle_g  +o(\eps).\label{expansion-1-1}
\end{align}
Simply writing, as before, $\nabla f$ in place of $\nabla_g f$ for a smooth function $f: \overline \Omega \to \R$ in the following, we also deduce from \eqref{expansion-2-app}
that \begin{align}
\langle \nabla_{g_\eps} f ,  \nabla_{g_\eps} h \rangle_{g_\eps} & = g_\eps^{\;\;ij}\: 
\partial_i f  \partial_j h \nonumber\\
&=\langle \nabla_g f, \nabla_g h\rangle_{g} -
\eps \langle \nabla_{\nabla h} V, \nabla f  \rangle_g  - \eps \langle \nabla_{\nabla f} V, \nabla h \rangle_g  +o(\eps) \label{expansion-2-1}
\end{align}
for smooth functions $f,h: \overline \Omega \to \R$. Moreover, the expansion is uniform when $f,h$ are taken from a bounded set in $C^1(\overline \Omega)$.
In order to establish (\ref{eq:zanger-generalized}), we now first prove that 
\begin{equation}
  \label{eq:zanger-generalized-leq}
\partial_\eps^+ \Big|_{\eps=0} \mu_2(\Omega_\eps) \le 
\min \Bigl \{\int_{\partial \Omega} (|\nabla u|^2- \mu_2(\Omega) u^2)\la V,\eta\ra_g\,d\sigma\::\: 
u \in L\Bigr \}.
\end{equation}
Let $u \in L$. From (\ref{eq:var-char-mu-2-rewritten}) it follows that 
$$
\mu_2(\Omega_\eps) \le \rho_u(\eps) := \frac{\int_{\Omega} | \nabla_{g_\eps} u|^2_{g_\eps}\,dv_{g_\eps}}{\int_{\Omega}[u- m(u,\eps)]^2\,d v_{g_\eps}}\qquad \text{for $|\eps| \le \eps_0$.} 
$$
Since also $\mu_2(\Omega)=\rho_u(0)$, we have 
$\partial_\eps^+ \Big|_{\eps=0} \mu_2(\Omega_\eps) \le \rho_u'(0)$, so the inequality (\ref{eq:zanger-generalized-leq}) follows once we have shown that 
\begin{equation}
  \label{eq:f-prime-0}
\rho_u'(0) = \int_{\partial \Omega} (|\nabla u|^2- \mu_2(\Omega) u^2) \langle V,\eta \rangle_g\,d\sigma \qquad \text{for $u \in L$.}  
\end{equation}
By expansions (\ref{eq:expansion-3-app}) and (\ref{expansion-2-1}), we have that  
\begin{equation}
\int_{\Omega} |\nabla_{g_\eps} u |_{g_\eps}^2 dv_{g_\eps} = \int_{\Omega} |\nabla u|_g^2\,dx + \eps 
\int_{\Omega} \Bigl(|\nabla u|^2 \div_g V-  2 \langle \nabla_{\nabla u} V, \nabla u  \rangle_g \Bigr)\,dx   +o(\eps) \label{expansion-2-1-1}
\end{equation} 
and therefore, via integration by parts, 
\begin{align*}
\partial_\eps &\Big|_{\eps=0} \int_{\Omega} |\nabla_{g_\eps} u |_{g_\eps}^2 dv_{g_\eps}
= \int_{\Omega} (|\nabla u|^2 \div_g V  - 2 \la \nabla u , \nabla_{\nabla u}V\ra_g \,dx\\
&= \int_{\partial \Omega} |\nabla u|^2 \langle V, \eta  \rangle_g d\sigma 
  - 2 \int_{\Omega} \Bigl(\langle \nabla_{\nabla u} \nabla u , V \rangle_g +
\langle \nabla u , \nabla_{\nabla u}V \rangle_g
 \Bigr) \,dx.
\end{align*} 
Since $\nabla u \cdot \eta  = 0$ on $\partial \Omega$ and $-\Delta u = \mu_2(\Omega) u$ in $\Omega$, we also have that 
\begin{align}
0 &= \int_{\partial \Omega} \langle \nabla u, V \rangle_g \langle \nabla u, \eta \rangle_g \,d\sigma = \int_{\Omega} \div \bigl( \langle \nabla u, V \rangle_g \nabla u\bigr) \,dx \nonumber \\ 
&= \int_{\Omega} \Bigl(\bigl \langle 
\nabla \langle \nabla u,V \rangle_g \nabla u \bigr \rangle_g +   \langle \nabla u, V \rangle_g \Delta u \Bigr)\,dx = \int_{\Omega} \Bigl(\partial_{\,\nabla u}\, \langle \nabla u,V \rangle_g - \mu_2(\Omega) u  \langle \nabla u, V \rangle_g \Bigr)\,dx\nonumber\\
&= \int_{\Omega} \Bigl( \langle \nabla_{\nabla u} \nabla u,V \rangle_g + \langle \nabla u,\nabla_{\nabla_u} V \rangle_g  - \mu_2(\Omega) u  \langle \nabla u, V 
\rangle_g \Bigr)\,dx  \label{int-by-parts}
\end{align}
and thus 
$$
\partial_\eps \Big|_{\eps=0} \int_{\Omega} |\nabla_{g_\eps} u |_{g_\eps}^2 dv_{g_\eps}= \int_{\partial \Omega} |\nabla u|_g^2 \langle V,\nu \rangle_g\, d\sigma 
  - 2\mu_2(\Omega) \int_{\Omega} u  \langle \nabla u, V \rangle_g \,dx.
$$
Using (\ref{eq:expansion-3-app}),  we also see that 
\begin{equation*}
\partial_\eps \Big|_{\eps=0} \int_{\Omega} u^2 dv_{g_\eps} = \int_{\Omega} u^2 \div_g V \,dx= \int_{\partial \Omega} u^2 \langle V, \eta  \rangle_g \,d\sigma - 2 \int_{\Omega} u \langle \nabla u \cdot V \rangle_g\,dx.
\end{equation*}
Moreover, since $m(u,0)=0$, we have $\partial_\eps \big|_{\eps=0}\bigl[ |\Omega_\eps| \, m^2(u,\eps)\bigr] =0$ and therefore 
\begin{align*}
\partial_\eps \Big|_{\eps=0} \int_{\Omega} [u- m(u,\eps)]^2\,d v_{g_\eps}&= \partial_\eps \Big|_{\eps=0} \Bigl(\int_{\Omega} u^2\,dx -
|\Omega_\eps| m^2(u,\eps)\Bigr)=\partial_\eps \Big|_{\eps=0} \int_{\Omega} u^2\,dv_{g_\eps}\\ 
&=\int_{\partial \Omega} u^2 \langle V, \eta \rangle_g \,d\sigma - 2 \int_{\Omega} u \langle \nabla u \cdot V \rangle_g\,dx
\end{align*}
by (\ref{eq:expansion-3-app}) and integration by parts. Combining the above identities, we find that 
\begin{align*}
\rho_u'(0) &= \Bigl(\int_{\Omega}u^2\,dx\Bigr)^{-2}\Bigl[\int_{\partial \Omega} |\nabla u|^2 \langle V, \eta \rangle_g \,d\sigma    - 2\mu_2(\Omega) \int_{\Omega} u  \langle \nabla u, V \rangle_g \,dx \\
&- \int_{\Omega}|\nabla u|^2\,dx \Bigl(\int_{\partial \Omega} u^2 \langle V \cdot \eta  \rangle_g \,d\sigma - 2 \int_{\Omega} u \langle \nabla u, V \rangle_g\,dx\Bigr)\Bigr]\\
&= \Bigl(\int_{\Omega}u^2\,dx\Bigr)^{-2}\Bigl[\int_{\partial \Omega} |\nabla u|^2  \langle V, \eta \rangle_g \,d\sigma    - 2 \mu_2(\Omega) \int_{\Omega} u   \langle \nabla u, V \rangle_g \,dx \\
&- \mu_2(\Omega) \Bigl(\int_{\partial \Omega} u^2 \langle V,\eta \rangle_g \,d\sigma - 2 \int_{\Omega} u \langle \nabla u, V \rangle_g \,dx\Bigr)\Bigr]= \int_{\partial \Omega} (|\nabla u|^2- \mu_2(\Omega) u^2) \langle V, \eta  \rangle_g \,d\sigma,
\end{align*}
as claimed in (\ref{eq:f-prime-0}). We thus conclude that (\ref{eq:zanger-generalized-leq}) holds.
Next, to show the opposite inequality in (\ref{eq:zanger-generalized}), we argue by contradiction. Hence we suppose that 
\begin{equation}
  \label{eq:zanger-generalized-geq-contra-1}
\liminf_{\eps  \to 0^+}\frac{\mu_2(\Omega_{\eps})-\mu_2(\Omega)}{\eps} < \kappa_\Omega:=
\min_{u \in L} \int_{\partial \Omega} (|\nabla u|^2- \mu_2(\Omega) u^2) \langle V,\eta \rangle_g \,d\sigma,
\end{equation}
which means there exists a sequence of positive numbers $\eps_k$, $k \in \N$ with $\eps_k \to 0$ and such that 
\begin{equation}
  \label{eq:zanger-generalized-geq-contra-2}
\lim_{k \to \infty}\frac{\mu_2(\Omega_{\eps_k})-\mu_2(\Omega)}{\eps_k} <  \kappa_{\Omega}.
\end{equation}
For the ease of notation, we simply write $\eps$ in place of $\eps_k$ in the following. Moreover, we let $u^\eps$ denote an $L^2$-normalized eigenfunction on $\Omega_\eps$ corresponding to the eigenvalue $\mu_2(\Omega_\eps)$. Using again the fact that the map $\tau_\eps : (M, g_\eps) \to (M,g)$ is an isometry, we find that the functions $u_\eps \in C^2(\overline \Omega)$, $u_\eps:= u^\eps \circ \tau_\eps$ satisfy 
$$
-\Delta_{g_\eps} u_\eps = \mu_2(\Omega_\eps) u_\eps \quad \text{in $\Omega$,}\qquad \partial_{\eta_\eps} u_\eps = 0 \quad \text{on $\partial \Omega$.}
$$
Here $\eta_\eps$ denotes the outer normal on $\partial \Omega$ with respect to the metric $g_\eps$. By elliptic regularity (using the fact that the coefficients of $g_\eps$ are locally uniformly Lipschitz), it follows that the sequence $(u_\eps)_\eps$ remains bounded in $C^{1,\alpha}(\overline{\Omega})$, and thus $u_\eps \to w$ in $C^1(\overline \Omega)$ after passing to a subsequence. Integrating by parts and 
using the expansions (\ref{eq:expansion-3-app}) and (\ref{expansion-2-1}) again, we thus infer that  
\begin{align*}
&\mu_2(\Omega_\eps) \int_{\Omega}u_\eps w\,d v_{g_\eps} = 
- \int_{\Omega} (\Delta_{g_\eps} u_\eps) w\,dx=\int_{\Omega} \langle \nabla_{g_\eps} u_\eps, \nabla_{g_\eps} w \rangle_{g_\eps} \,dv_{g_\eps}\\
&=  \int_{\Omega} \nabla u_\eps  \cdot \nabla w  \,dx + \eps \int_{\Omega} \Bigl( \langle \nabla u_\eps, \nabla w \rangle_g \div_g V -  \langle \nabla_{\nabla u_\eps} V, \nabla u_\eps \rangle_g -  \langle 
\nabla_{\nabla w} V, \nabla u_\eps \rangle_g\Bigr) \,dx + o(\eps)\\
&=  \mu_2(\Omega) \int_{\Omega} u_\eps w  \,dx + \eps \int_{\Omega} \Bigl(|\nabla w|^2\, \div_g V -2  \langle 
\nabla_{\nabla w} V, \nabla w  \rangle_g\Bigr) \,dx + o(\eps)\\
&=  \mu_2(\Omega) \Bigl[\int_{\Omega} u_\eps w dv_{g_\eps} - \eps \int_\Omega u_\eps w\, \div_g V dx\Bigr] + \eps \int_{\Omega} \!\Bigl(|\nabla w|^2\, \div_g V -  2  \langle 
\nabla_{\nabla w} V, \nabla w  \rangle_g \Bigr)dx + o(\eps)
\end{align*}
Using also that 
$$
\int_{\Omega} u_\eps w \,dv_{g_\eps} =
\int_{\Omega} w^2 \,dx +o(1)= 1 + o(1) \quad \text{and}\quad 
\int_\Omega u_\eps w\, \div_g V\,dx = \int_\Omega w^2 \div_g V\,dx + o(1) 
$$
we conclude that 
\begin{align*}
&\mu_2(\Omega_\eps) = \mu_2(\Omega) \\ 
&+ \eps \Bigl(\int_{\Omega} u_\eps w \,dv_{g_\eps}\Bigr)^{-1}\Bigl(
\int_{\Omega} \Bigl( |\nabla w|^2 \div_g V- \mu_2(\Omega)w^2 \div_g V -
2  \langle 
\nabla_{\nabla w} V, \nabla w  \rangle_g \Bigr)\,dx +o(\eps)\\
&= \mu_2(\Omega) + \eps \int_{\Omega} \Bigl( |\nabla w|^2 \div_g V- \mu_2(\Omega)w^2 \div_g V - 2  \langle 
\nabla_{\nabla w} V, \nabla w  \rangle_g  \Bigr)\,dx +o(\eps).
\end{align*}
Integrating by parts, we thus find that 
\begin{align*}
&\frac{\mu_2(\Omega_\eps)- \mu_2(\Omega)}{\eps}= 
\int_{\Omega} \Bigl(|\nabla w|^2 \div_g V- \mu_2(\Omega)w^2 \div_g V  -2  \langle 
\nabla_{\nabla w} V, \nabla w  \rangle_g   \Bigr)\,dx +o(1)\\
&= \int_{\partial \Omega} \Bigl(|\nabla w|^2 - \mu_2(\Omega)w^2 \Bigr) \langle V, \eta \rangle_g \,d\sigma \\
&+ 2 \int_{\Omega} \Bigl(\mu_2(\Omega) w \langle \nabla w,  V \rangle_g -   \langle 
\nabla_{\nabla w} V, \nabla w  \rangle_g  - \langle \nabla_{\nabla w} \nabla w,V \rangle_g \Bigr)\,dx +o(1)\\
&= \int_{\partial \Omega} \Bigl(|\nabla w|^2 - \mu_2(\Omega)w^2 \Bigr) \langle V, \eta  \rangle_g \,d\sigma +o(1) \ge \kappa_\Omega + o(1),
\end{align*}
where we used \eqref{int-by-parts} with $w$ in place of $u$.
Recalling that this holds for a subsequence of the sequence $(\eps_k)_k$ for which we assumed (\ref{eq:zanger-generalized-geq-contra-2}), we thus get a contradiction. We conclude that both $\le$ and $\ge$ holds in (\ref{eq:zanger-generalized}), and thus the proof is finished.
\end{proof}

\begin{Corollary}
\label{overdetermined-corollary}
Let $\Omega$ and $L$ be as in Proposition \ref{zanger-formula}. Then we have the following. 
\begin{enumerate}
\item[(i)] If $\Omega$ is a local minimum with respect to domain variations, then the quantity $|\nabla u|^2- \mu_2(\Omega) u^2$ is constant on $\partial \Omega$  for all $u \in L$. In particular, $\Omega$ is a constraint critical point for $\mu_2$ in strong sense.  
\item[(ii)] If $\Omega$ is a local maximum with respect to domain variations, then $\Omega$ is a constraint critical point for $\mu_2$ in weak sense.
\end{enumerate}
\end{Corollary}

\begin{proof}
(i) If suffices to show that 
\begin{equation}
  \label{eq:suffice-minimum}
\int_{\partial \Omega} \Bigl(|\nabla u|^2 - \mu_2(\Omega)u^2 \Bigr)h\,d\sigma = 0 
\end{equation}
for every $u \in L$ and every $C^1$-function $h: \partial \Omega \to \R$ with 
$\int_{\partial \Omega} h d \sigma = 0$. Fix such a function $h$, and consider the corresponding $C^1$-curve $(-\eps_0,\eps_0) \to \cV^1(\M), \quad t \mapsto V_t$ given by Lemma~\ref{sec:technical-lemma-1}(i). Combining the assumption with Proposition~\ref{zanger-formula} and Lemma~\ref{sec:technical-lemma-1}(i), we then deduce that 
$$
0 \le \partial_t^+ \Big|_{t=0}\: \mu_2(\Omega_{\text{\tiny $V_t$}}) = \min \Bigl \{\int_{\partial \Omega} (|\nabla u|^2- \mu_2(\Omega) u^2) h \,d\sigma\::\: 
u \in L\Bigr \}.
$$
Replacing $h$ by $-h$, we then also deduce that   
$$
\max \Bigl \{\int_{\partial \Omega} (|\nabla u|^2- \mu_2(\Omega) u^2)h\,d\sigma\::\: 
u \in L\Bigr \} \le 0,
$$
and thus (\ref{eq:suffice-minimum}) follows.\\
(ii) By the same argument as in the proof of (i), we see that 
$$
\min \Bigl \{\int_{\partial \Omega} (|\nabla u|^2- \mu_2(\Omega) u^2)h \,d\sigma\::\: 
u \in L\Bigr \} \le 0 
$$
for all $C^1$-functions $h: \partial \Omega \to \R$ with $\int_{\partial \Omega}h d \sigma = 0$. By density, this yields, 
\begin{equation}
  \label{eq:cons-maximum}
\min \Bigl \{\int_{\partial \Omega} (|\nabla u|^2- \mu_2(\Omega) u^2)h\,d\sigma\::\: u \in L\Bigr \} \le 0 \quad \text{for $h\in L^2(\de\O)$ with $\int_{\de\O} h\,d\s=0$.}
\end{equation}
We now consider the set 
$K \subset L^2(\partial \Omega)$ given as the convex hull of the set  
$$
K_0:= \Bigl \{ \Bigl(|\nabla u|^2- \mu_2(\Omega) u^2\Bigr)\Big|_{\partial \Omega} \::\: u \in L \Bigr \}.
$$
Since $K_0$ is a compact set contained in the finite dimensional space $E_0 \subset L^2(\partial \Omega)$ spanned by 
\be 
\label{span-functions} 
\Bigl(|\nabla u_i|^2- \mu_2(\Omega) u_i^2 \Bigr)\Big|_{\partial \Omega}, \qquad \Bigl(\langle \nabla u_i, \nabla u_j \rangle_g - \mu_2(\Omega) u_i u_j \Bigr)\Big|_{\partial \Omega}, \qquad i,j= 1,\dots, \ell,
\ee
where $u_1,\dots,u_\ell$ denotes a basis of the eigenspace corresponding to $\mu_2(\Omega)$, it follows from Carath\'eodory's theorem that $K$ is compact as well. Let $P \subset L^2(\partial \Omega)$ denote the one-dimensional subspace of constant functions. We claim that 
\begin{equation}
  \label{eq:maximum-sufficient}
K \cap P \not = \varnothing.
\end{equation}
For this we consider the 
the finite dimensional space $E= E_0 + P \subset  L^2(\partial \Omega)$, which is a Hilbert space with  the induced scalar product of $L^2(\de\O)$. Suppose by contradiction that $K \cap P = \varnothing$. Then there exists a convex relatively open neighborhood $\tilde K$ of $K$ in $E$ such that $\tilde K \cap P = \varnothing$. 
By Mazur's separation theorem, there thus exists some function $\tilde h \in E $ such that 
$$
\int_{\de\O} \tilde h  w\, d\s = 0 \quad \text{for $w \in P$}\qquad \text{and}\qquad\int_{\de\O} \  \tilde h w \, d\s > 0 \quad \text{for $w \in  \tilde K $.}
$$
In particular, 
$$
\int_{\de\O} \tilde h  \, d\s  = 0
\qquad \text{and}\qquad \int_{\partial \Omega} (|\nabla u|^2- \mu_2(\Omega) u^2)\tilde h\,d\sigma > 0 \;\qquad \text{for all $u \in L$,}
$$
which contradicts (\ref{eq:cons-maximum}) since $L$ is compact. Hence we conclude that (\ref{eq:maximum-sufficient}) holds. Consequently, there exists $m \in \N$, $\lambda_1,\dots,\lambda_m  \ge 0$ with $\sum \limits_{k=1}^m \lambda_k = 1$ and $u_1,\dots,u_m \in L$ 
such that 
$$
\sum_{k=1}^m \lambda_k (|\nabla u_k|^2- \mu_2(\Omega) u_k^2) = \lambda \qquad \text{on $\partial \Omega$}
$$
with a constant $\lambda \in \R$. Without loss of generality, we may assume here that $\lambda_k \not = 0$ and $u_k \not = 0$ for $k=1,\dots,m$. Replacing $u_k$ by $\sqrt{\lambda_k} u_k$, we thus obtain that 
\begin{equation}
  \label{eq:-msum}
\sum_{k=1}^m (|\nabla u_k|^2- \mu_2(\Omega) u_k^2) = \lambda \qquad \text{on $\partial \Omega$},
\end{equation}
which means that $\Omega$ is a constrained critical point for $\mu_2$ in weak sense, as claimed. 
\end{proof}

\begin{Remark}
\label{remark-caratheodory}{\rm 
The above proof is, to some extend, inspired by similar arguments in \cite{Nadi} and \cite{Soufi-Ilias}. An inspection of the proof shows that the number $m$ in (\ref{eq:-msum}) can be chosen less than or equal to $\frac{\ell(\ell+1)}{2}+1$, where $\ell$ is the dimension of the eigenspace $L$ corresponding to $\mu_2(\Omega)$. This follows from Carath\'eodory's theorem and the fact that the dimension of the space $E_0$ spanned by the functions in \eqref{span-functions} is less than or equal to $\frac{\ell(\ell+1)}{2}$. It would be interesting to know whether this bound on $m$ is optimal.}     
\end{Remark}

The following Proposition is the second main step in the proofs of Theorem \ref{sec:main-theorem-1}(ii),(iii) and Theorem \ref{sec:special-theorem-1}. 

\begin{Proposition}
\label{proposition-implication-strong-sense}
Let $\Omega \in \cO^2(\M)$ be such that there exists a nontrivial solution of the overdetermined problem 
\begin{equation}
\label{eq:overdetermined-1}
\left \{
\begin{aligned}
-\D_g u&=\mu_2(\Omega) u &&\qquad \text{in $\O$},\\
\partial_\eta u&=0,\quad |\nabla u|^2 -\mu_2(\Omega) u^2 = \lambda    
&&\qquad \text{on $\de \O$,}
  \end{aligned}
\right.
\end{equation}
for  some constant $\lambda \in \R$. Then 
\begin{equation}
  \label{eq:usquare-const}
u^2 \equiv - \frac{\lambda}{\mu_2(\Omega)} >0 \qquad \text{on $\partial \Omega$.} \end{equation}
  In addition, 
 there exists $\eps>0$ such that $\mu_2(\Omega_\V) < \mu_2(\Omega)$ for every 
admissible deformation field $V$ for $\Omega$ which is locally volume preserving in $U_\eps(\partial \Omega)$ and such that $\Omega_\V \not = \Omega$. 
\end{Proposition}

\begin{proof}
Let $u$ be a nontrivial solution of (\ref{eq:overdetermined-1}).   To see this,  choose $x_1,x_2 \in \partial \Omega$ such that $u^2(x_1)= \max
\limits_{\partial \Omega} u^2$ and $u^2(x_2)= \min \limits_{\partial
  \Omega}u^2$, so that $\n u^2(x_1)= 2u(x_1)\n u(x_1)=0$. By unique continuation,
we know that $u^2(x_1) \not = 0$, so that $\n u(x_1) =0$, yielding  $u^2(x_1)=-  \frac{\lambda}{\mu_2(\O)}>0$. This latter property and the fact that $\n u^2(x_2)=2u(x_2)\n u(x_2)=0$ imply that $\n u(x_2) = 0$ and thus  $u^2(x_2)=-  \frac{\lambda}{\mu_2(\O)}$. This proves (\ref{eq:usquare-const}).\\
In the following, we put $\lambda_0:= - \frac{\lambda}{\mu_2(\Omega)}$ for the constant value of $u^2$ on $\partial \Omega$. Moreover, 
we let $\Delta_{\text{\tiny $\partial \Omega$}}$ denote the Laplace-Beltrami operator on the $N-1$-dimensional submanifold $\partial \Omega$ and $H_{\text{\tiny $\partial \Omega$}}$ the mean curvature of $\partial \Omega$. Since $\Delta_{\text{\tiny $\partial \Omega$}} u \equiv 0$ and $\partial_\eta  u = 0$ on $\partial \Omega$, we find that 
$$
\partial_{\eta  \eta } [u^2]  = 2 u \partial_{\eta \eta } u =
2 u\Bigl(\Delta u - \Delta_{\text{\tiny $\partial \Omega$}} u - H_{\text{\tiny $\partial \Omega$}} \partial_\eta  u\Bigr)
=2 u \Delta u = -2 \mu_2(\Omega) u^2 =2 \lambda < 0 
$$
on $\partial \Omega$. Consequently, there exists $\eps>0$ such that 
\begin{equation}
  \label{eq:u-square-lambda-0-ineq}
u^2 < \lambda_0 \qquad \text{in $U_\eps(\partial \Omega) 
\cap \Omega$.}
\end{equation}
Next, we decompose $\partial \Omega$ into the compact subsets 
$\Gamma_\pm := \{x \in \partial \Omega\::\: u(x)= \pm \sqrt{\lambda_0}\}.$
By making $\eps>0$ smaller if necessary, we can then achieve that 
$$
U_{\eps}(\Gamma_+) \cap U_{\eps}(\Gamma_-) = \varnothing
$$
and 
\begin{equation}
  \label{eq:u-sign-ineq}
0< u < \sqrt{\lambda_0} \quad \text{in $U_{\eps}(\Gamma_+)\cap \Omega$}, \qquad 
-\sqrt{\lambda_0}  < u<  0 \quad \text{in $U_{\eps}(\Gamma_-)\cap \Omega$}.
\end{equation}
In the following, we fix an admissible deformation field $V \in \cV(\M)$ for $\Omega$ which is locally volume preserving in $U_\eps(\partial \Omega)$ and such that $\Omega_V \not = \Omega$. To complete the proof, we need to show that 
\begin{equation}
  \label{eq:sufficient-mu-2-local-max}
\mu_2(\Omega_V) < \mu_2(\Omega).  
\end{equation}
For this we define the function $w \in C^1(U_{\eps}(\Omega))$ by 
$$
w(x)= 
\left \{
  \begin{aligned}
   &u(x),&&\qquad x \in \Omega,\\
   &+\sqrt{\lambda_0},&&\qquad x \in U_{\eps}(\Gamma_+) \setminus \Omega,\\
   &-\sqrt{\lambda_0},&&\qquad x \in U_{\eps}(\Gamma_-) \setminus \Omega.\\
  \end{aligned}
\right.
$$
Since $\Omega_\V \subset U_\eps(\ov \Omega)$, we may use $w$ in the variational characterization of $\mu_2(\Omega_\V)$ to deduce that 
\begin{equation}
  \label{eq:var-char-w-est}
\mu_2(\O_V) \le \frac{\int_{\O_\V}|\n w|^2\,dx}{\int_{\O_\V} (w - m(w))^2\,dx}\qquad \text{with}\quad m(w) := \frac{1}{|\Omega_\V|}\int_{\Omega_\V} w\,dx.
\end{equation}
Since $|\Omega_\V|=|\Omega|$, we have 
\begin{equation}
   \label{eq:-star-additional}
\int_{\O_\V} (w - m(w))^2\,dx = \int_{\O_\V} w^2\,dx - \frac{1}{|\Omega_\V|}
\Bigl(\int_{\O_\V} w\,dx\Bigr)^2 = \int_{\O_\V} w^2\,dx - \frac{1}{|\Omega|} \Bigl(\int_{\O_\V} w\,dx\Bigr)^2. 
\end{equation}
Moreover, since $|\Omega_\V \setminus \Omega|= |\Omega \setminus \Omega_\V|$, we have that
\begin{align}
\int_{\O_\V}& w^2\,dx = \int_{\O_\V \setminus \O}w^2\,dx + \int_{\O} w^2\,dx - \int_{\O \setminus \O_\V} w^2\,dx \nonumber\\
&= \lambda_0 |\O_\V \setminus \O| + \int_{\O} u^2\,dx - 
\int_{\O \setminus \O_\V} u^2\,dx= \int_{\O} u^2\,dx + \int_{\O \setminus \O_\V} (\lambda_0- u^2)\,dx.\label{int-w-sqrt-est}
\end{align}
Furthermore, since 
$$
\int_{\Omega}u \,dx = -\frac{1}{\mu_2(\Omega)} \int_{\Omega}\Delta u \,dx = \int_{\partial \Omega} u_\nu\,d\sigma = 0,
$$
we find that
\begin{align}
\int_{\O_\V} w\,dx &= \int_{\Omega_\V \setminus \O} w\,dx +  \int_{\Omega_\V \cap \O} u = \int_{\Omega_\V \setminus \O} w\,dx - \int_{\Omega \setminus \O_\V} u \,dx\\ 
&= \sum_{i=\pm}\Bigl( \int_{(\Omega_\V \setminus \O) \cap U_\eps(\Gamma_i)} w\,dx - \int_{(\Omega \setminus \O_\V)  \cap U_\eps(\Gamma_i)} u \,dx\Bigr) \nonumber\\
&= \sqrt{\lambda_0} \Bigl(\bigl|(\Omega_\V \setminus \O) \cap U_\eps(\Gamma_+)\bigr|-\bigl|(\Omega_\V \setminus \O) \cap U_\eps(\Gamma_-)\bigr|\Bigr) - \sum_{i=\pm}\: \int_{(\Omega \setminus \O_\V)  \cap U_\eps(\Gamma_i)} u \,dx
  \nonumber\\
&= \int_{(\Omega \setminus \O_\V)  \cap U_\eps(\Gamma_+)}\!\!\!\!\!\!(\sqrt{\lambda_0}- |u|) \,dx \quad  - \int_{(\Omega \setminus \O_\V)  \cap U_\eps(\Gamma_-)}\!\!\!\!\!\!(\sqrt{\lambda_0}-|u|) \,dx
 \label{w-average-est}
\end{align}
Here we used (\ref{eq:u-sign-ineq}) and the fact that $|(\Omega_\V \setminus \O) \cap U_\eps(\Gamma_\pm)|= |(\O \setminus \O_\V ) \cap U_\eps(\Gamma_\pm)|$. Applying the Cauchy Schwarz inequality to the RHS of (\ref{w-average-est}), we deduce that 
\begin{align}
\Bigl(\int_{\O_\V} w\,dx\Bigr)^2 &\le \Bigl(|(\Omega \setminus \O_\V)  \cap U_\eps(\Gamma_+)|+|(\Omega \setminus \O_\V)  \cap U_\eps(\Gamma_-)|\Bigr) \sum_{i=\pm} \:\int_{(\Omega \setminus \O_\V)  \cap U_\eps(\Gamma_i)}\!\!\!\!\!\!\!(\sqrt{\lambda_0}- |u|)^2 \,dx \nonumber\\
&=|\Omega \setminus \O_\V| \int_{\Omega \setminus \O_\V}(\sqrt{\lambda_0}-|u|\bigr)^2 \,dx.\label{w-average-square-est}
\end{align}
Combining (\ref{eq:-star-additional}), (\ref{int-w-sqrt-est}) and (\ref{w-average-square-est}), we find that  
\begin{align}
\int_{\O_\V} (w - m(w))^2\,dx &\ge  \int_{\O} u^2\,dx + \int_{\O \setminus \O_\V} \Bigl[(\lambda_0- u^2) -(\sqrt{\lambda_0}-|u|)^2\Bigr]\,dx \nonumber\\
&= \int_{\O} u^2\,dx + 2 \int_{\O \setminus \O_\V} |u|\bigl(\sqrt{\lambda_0}-|u|\bigr)\,dx > \int_{\O} u^2\,dx,\label{w-minus-w-average-1}
\end{align}
where the last inequality follows from (\ref{eq:u-square-lambda-0-ineq}) and the fact that $|\Omega \setminus \O_\V|>0$ by assumption. Since $w$ is constant on $\Omega_\V \setminus \Omega$, we also have that
\begin{align}
\int_{\O_\V}|\n w|^2\,dx &= \int_{\O_\V \setminus \O}|\n w|^2\,dx + \int_{\O}|\n w|^2\,dx - \int_{\O \setminus \O_\V}|\n w|^2\,dx \nonumber\\   
&= \int_{\O}|\n u|^2\,dx - \int_{\O \setminus \O_\V}|\n u|^2\,dx \le \int_{\O}|\n u|^2\,dx. \label{grad-w-est}  
\end{align}
Combining (\ref{eq:var-char-w-est}), (\ref{w-minus-w-average-1}) and (\ref{grad-w-est}), we finally conclude that  
$$
\mu_2(\Omega_\V) < \frac{\int_{\O}|\n u|^2\,dx}{\int_{\O_\V} u^2\,dx} = \mu_2(\Omega). 
$$
We thus have (\ref{eq:sufficient-mu-2-local-max}), as required.
\end{proof}

\begin{Corollary}
\label{sec:vari-zang-dom-final-corol}
Under the assumptions of Proposition~\ref{proposition-implication-strong-sense} , $\Omega$ is not a constrained local minimum for $\mu_2$ unless $\Omega= \M$. Moreover, it is a strict local maximum if $\partial \Omega$ is connected. 
\end{Corollary}

\begin{proof}
Let $\eps >0$ be given by Proposition~\ref{proposition-implication-strong-sense}. If $\Omega \not = \M$, then by Lemma~\ref{sec:technical-lemma-1}(ii) there exists an admissible deformation field $V \in \cV(\M)$ for $\Omega$ which satisfies $\|V\|_{C^1}<\eps$, is locally volume preserving in $U_\eps(\partial \Omega)$ and such that $\Omega_V \not = \Omega$. Moreover, Proposition~\ref{proposition-implication-strong-sense} yields that $\mu_2(\Omega_\V) < \mu_2(\Omega)$ in this case. Hence $\Omega$ is not a constrained local minimum for $\mu_2$.\\
Moreover, if $\partial \Omega$ is connected, then there exists $\eps_1= \eps_1(\eps)>0$ such that every admissible deformation field $V \in \cV(\M)$ for $\Omega$ with $\|V\|_{C^1}< \eps_1$ is also locally volume preserving in 
$U_\eps(\partial \Omega)$, and thus Proposition~\ref{proposition-implication-strong-sense} yields that $\mu_2(\Omega_\V) < \mu_2(\Omega)$ if $\Omega_\V \not = \Omega$. This ends the proof.    
\end{proof}

\begin{proof}[Proof of Theorem~\ref{sec:main-theorem-1} (completed)]
Part (i) is already contained in Corollary~\ref{overdetermined-corollary}(ii), and Parts (ii) and (iii) follows directly from Proposition~\ref{proposition-implication-strong-sense}. 
\end{proof}

\begin{proof}[Proof of Theorem~\ref{sec:special-theorem-1} (completed)]
The result is already contained in Proposition~\ref{proposition-implication-strong-sense}.  
\end{proof}

\section{The case of cylindrical manifolds}\label{s:cylinder-case}

In this section, we restrict our attention to the case $\M$ is a cylindrical manifold of the form $\M:=\R^k \times \cN$, where 
$(\cN,g_\cN)$ is a closed connected manifold and $\M$ is endowed with the product metric $g= g_{eucl} \otimes g_{\cN}$. 

In the following, we let $B^k \subset \R^k$ denote the unit ball. 
As noted already in the introduction, $\mu_2(B^k)$ is of multiplicity $N$ with corresponding eigenfunctions $x \mapsto \vp(|x|)\frac{x_i}{|x|}$, $i=1,\dots, k$, 
where $\vp$ is the unique solution of the boundary value problem 
$$
 \vp''+\frac{k-1}{t}\vp'+\left(\mu_2(B)-\frac{k-1}{t^2}
\right)\vp=0,\quad t\in (0,1),\quad \vp(0)=\vp'(1)=0.
$$
The function $ \vp$  and the eigenvalue $\mu_2(B)$ can be characterized 
via $J_{k/2}$, the Bessel function of the first kind of order $k/2$, see e.g. \cite{lebedew}.
Indeed, $\sqrt{\mu_2(B)}$ is the first positive zero of the derivative
of  $t \mapsto t^{(2-k)/2} J_{k/2}( t)$, and $\vp$ is a scalar multiple of the function 
\begin{equation}
  \label{eq:42}
t \mapsto g(t)=t^{(2-k)/2} J_{k/2}(\sqrt{\mu_2(B)} t).
\end{equation}
As a consequence of these facts, the cylindrical domain
\be\label{eq:defG}
\Omega_r:=\{(t,x) \in \M \,:\, |t| \le r, ,\, x \in \cN\} \subset \M, \qquad r>0
\ee
admits the Neumann eigenvalue $\mu^r:= \frac{\mu_2(B^k)}{r^2}$ with eigenfunctions 
\begin{equation}
  \label{eq:eigenfunctions-cylinder}
u_r^i: \overline \Omega_r \to \R,\quad u_r^i(t,x):= \vp(\frac{|t|}{r})\frac{t_i}{|t|},\qquad i=1,\dots,k.
\end{equation}
In particular, for $k=1$ we have $\mu_2(B^1)= \frac{\pi^2}{4}$, and there is only one function of the type (\ref{eq:eigenfunctions-cylinder}), up to a constant factor, given by 
\begin{equation}
  \label{eq:k=1-r-eigenfunction}
u_r : \M \to \R, \qquad u_r(t,x)= \sin(\frac{\pi}{2r} t).  
\end{equation}
The following observation is the first step in the proof of Theorem~\ref{sec:weinberger-modified}, and it is closely related to Weinberger's euclidean isoperimetric inequality for $\mu_2$ in \cite{Wein}. 

\begin{Proposition}
\label{sec:introduction-1}
Let $r>0$. If $\Omega \in \cO^1(\M)$ satisfies
$|\Omega|=|\Omega_r|$, then $\mu_2(\Omega) \le \mu^r$ with equality if and only if
$\Omega=\Omega_r$ up to translation in $t$-direction.
\end{Proposition}

\begin{proof}
The proof is modeled on Weinberger's argument in \cite{Wein}. 
Consider the function $G:[0,\infty) \to \R$ defined by $G(\tau)=\vp(\frac{\tau}{r})$ for $\tau \le r$ and $G(\tau)=\vp(1)$ for $\tau \ge r$. Moreover, consider the continuous vector field 
$$
V: \R^k \to \R^k,  \qquad V(y) = \int_{\Omega} G\bigl(|t-y|\bigr) \frac{t-y}{|t-y|}d(t,x).
$$
Since $\Omega$ is bounded, we have 
$\frac{V(y)}{|V(y)|} = -\frac{y}{|y|}+ o(1)$ as $|y| \to \infty$. Hence Brower's fixed point implies that $V$ has a zero, and without loss we may, by translation $\Omega$ in the $t$-variables if necessary, assume that $V(0)=0$. Consequently, the restrictions of each of the functions 
$$
v_i: \M \to \R, \qquad v_i(t,x)= G(|t|)\frac{t_i}{|t|}
$$
to $\Omega$ belongs to $H^1(\Omega)$ and satisfies $\int_\Omega V_i\,dx = 0$. Therefore (\ref{eq:characterization-mu2}) implies that   
\begin{align}
 \nonumber 
\mu_2(\Omega) \int_{\Omega}G^2(|t|)\,d(t,x) &= \sum_{i=1}^k \mu_2(\Omega) 
\int_{\Omega} v_i^2\,d(t,x)\\
 &\le \sum_{i=1}^k \int_{\Omega} |\nabla v_i|^2\,d(t,x) = \int_{\Omega}
H(|t|)\,d(t,x) \label{eq:test-1}
\end{align}
with $H:(0,\infty) \to \R$ given by $H(\tau)=[G'(\tau)]^2 + \frac{G^2(\tau)}{\tau^2}$. As noted in 
\cite{Wein}, $G$ and $H$ are nonnegative functions such that $G$ is increasing and $H$ is strictly decreasing. Consequently, since $|\Omega|= |\Omega_r|$, we have that  
\begin{align*}
\int_{\Omega}&G^2(|t|)\,d(t,x) = \int_{\Omega_r}G^2(|t|)\,d(t,x) + 
\int_{\Omega \setminus \Omega_r}G^2(|t|)\,d(t,x) -\int_{\Omega_r \setminus \Omega}G^2(|t|)\,d(t,x)\\ 
 &\ge \int_{\Omega_r}G^2(|t|)\,d(t,x) + G(1) \Bigl( |\Omega \setminus \Omega_r| - 
|\Omega_r \setminus \Omega|\Bigr) \ge  \int_{\Omega_r}G^2(|t|)\,d(t,x),
\end{align*}
and, similarly, 
\begin{equation}
  \label{eq:ineq-H}
\int_{\Omega} H(|t|)\,d(t,x) \le \int_{\Omega_r}
H(|t|)\,d(t,x).
\end{equation}
Moreover, equality holds in (\ref{eq:ineq-H}) if and only if $|\Omega_r \setminus \Omega|= |\Omega \setminus \Omega_r|= 0$, i.e. if $\Omega= \Omega_r$. Consequently, we have that 
\begin{equation}
  \label{eq: finalineq-Omega-Omegar}
\mu_2(\Omega) \le \frac{\int_{\Omega_r}G^2(|t|)\,d(t,x)}{\int_{\Omega_r}
H(|t|)\,d(t,x)}
\end{equation}
with equality if and only if $\Omega= \Omega_r$. Since equality holds in (\ref{eq:test-1}) when $\Omega$ is replaced by $\Omega_r$ and $\mu_2(\Omega)$ by $\mu^r$, the right hand side of (\ref{eq: finalineq-Omega-Omegar}) equals $\mu^r$. Thus the proof is finished.
\end{proof}

We may now finish the

\begin{proof}[Proof of Theorem~\ref{sec:weinberger-modified}]
Let $r>0$. By separation of variables, only two cases may occur:\\
{\bf \em Case 1:} $\mu_2(\Omega_r)= \mu_2(\cN)$, and at least one associated eigenfunction on $\Omega_r$ is of the form $(t,x) \mapsto w(x)$, where $w \in C^2(\cN)$ is an eigenfunction corresponding to $\mu_2(\cN)$.\\
{\bf \em Case 2:} $\mu_2(\Omega_r)= \mu^r$, and the functions given in (\ref{eq:k=1-r-eigenfunction}) are contained in the associated eigenspace.
\\
Clearly, Case 2 occurs if and only if $\mu^r = \frac{\mu_2(B^k)}{r^2} \le \mu_2(\cN)$, which holds if and only if $|\Omega_r| = \omega_k |\cN| r^k$ is larger than or equal to the critical volume given in Theorem~\ref{sec:weinberger-modified}. Thus, if $v  \ge v_c$ is fixed and $r= \Bigl(\frac{v}{\omega_k|\cN|}\Bigr)^{\frac{1}{k}}$, then Proposition~\ref{sec:introduction-1} yields that 
$$
\mu_2(\Omega) \le \mu^r = \mu_2(\Omega_r) \qquad \text{for every domain $\Omega \in \cO^1(\M)$ with $|\Omega|= v$.}
$$
Moreover, equality holds if and only if $\Omega=\Omega_r$ up to translation in the $t$-variable.
\end{proof}

We now turn to the overdetermined boundary value problem
\begin{equation}
\label{eq:overdetermined}
\left \{
\begin{aligned}
-\D_g u&=\mu_2(\Omega) u &&\qquad \text{in $\O$},\\
\partial_\eta u&=0, \quad
|\nabla u|^2 -\mu_2(\Omega) u^2 = \lambda    
&&\qquad \text{on $\de \O$}.
  \end{aligned}
\right.
\end{equation}
The remainder of this section will be devoted to the proof of Theorem~\ref{sec:introduction-3-intro}, which we restate here for the reader's convenience.

\begin{Theorem}
\label{sec:introduction-3}
Let $\Omega \in \cO^2(\M)$ be a domain such that the overdetermined problem 
\eqref{eq:overdetermined} admits a solution for some constant $\lambda \in \R$. Then $k=1$ and $\Omega=\Omega_r$ for some $r \ge \Bigl(\frac{\mu(B^k)}{\mu(\cN)}\Bigr)^{\frac{1}{2}}$ up to translation in the $t$-variable.
\end{Theorem}

\begin{proof}
Let $u$ be a nontrivial solution of (\ref{eq:overdetermined}). By Theorem~\ref{sec:introduction-3} we have that $\lambda<0$ 
and $u^2 \equiv \lambda_0:=- \frac{\lambda}{\mu_2(\O)}$ on $\partial \Omega$. So $u$ is locally constant and nonzero on $\partial \Omega$.
Next we consider some unit vector $\sigma \in \R^k$ and the directional derivative 
$$
u_{{\sigma}} = \partial_{\text{\tiny $({\sigma},0)$}}u : \overline{\Omega} \to \R,\qquad 
u_{{\sigma}}(t,x)= \lim_{\eps \to 0} \frac{u(t+ \eps {\sigma},x)-u(t,x)}{\eps}.
$$
We claim that 
\begin{equation}
  \label{eq:claim-u-sigma}
\text{for every unit vector $\sigma \in \R^k$ we have $u_\sigma>0$ in $\Omega$ or $u_\sigma<0$ in $\Omega$.}  
\end{equation}
Indeed, differentiating the first equation in (\ref{eq:overdetermined}) and
recalling that $\nabla u \equiv 0$ on $\partial \Omega$, we
see that $u_{{\sigma}}$ solves 
\be
\label{eq:dirichlet-partial-derivative}
\left \{
\begin{aligned}
-\D_g u_{{\sigma}}&=\mu_2(\O) u_{{\sigma}}&&\qquad \textrm{in } \O,\\
u_{{\sigma}}& = 0&&\qquad \textrm{on } \de \O,\\
\end{aligned}
\right.
\ee
If we now suppose by contradiction that $u_{{\sigma}}$ changes sign, then the second Dirichlet eigenvalue $\lambda_2(\Omega)$  of $-\D_g$
on $\Omega$ is less than or equal to $\mu_2(\Omega)$. On the other hand, the variational characterization (\ref{eq:characterization-mu2}) gives rise to the inequality $\mu_2(\Omega) \le \lambda_2(\Omega)$, and so equality holds. But then a nontrivial linear combination $v$ of the positive and negative part of $u_{{\sigma}}$ is a corresponding Neumann eigenfunction which thus solves the equation $-\D v = \mu_2(\Omega) v$ in $\Omega$ together with homogeneous  Dirichlet and Neumann boundary  
conditions on $\partial \Omega$. This is impossible by unique continuation. Hence $u_\sigma$ does not change sign.

Next we suppose by contradiction that  $u_\sigma \equiv 0$. Let then $(t,x) \in \Omega$, and let ${\cal C}$ be the connected component of the set $\{(t+\tau {\sigma},x) \::\: \tau \in \R\}\cap \Omega$ which contains $(t,x)$. Then $u$ is constant on ${\cal C}$. Since $\overline 
{\cal C} \cap \partial \Omega \not = \varnothing$, we thus conclude that $u(t,x) = \sqrt{\lambda_0}$ or $u(t,x)= -\sqrt{\lambda_0}$. Since this holds for every point$(t,x) \in \Omega$, the connectedness of $\Omega$ implies that $u \equiv \sqrt{\lambda_0}$ in $\Omega$ or $u \equiv -\sqrt{\lambda_0}$ in $\Omega$, which contradicts the first equation in (\ref{eq:overdetermined}). Consequently we have $u_\sigma \not \equiv 0$. Now \eqref{eq:dirichlet-partial-derivative} and the strong maximum principle imply that $u_\sigma>0$ in $\Omega$ or $u_\sigma<0$ in $\Omega$, as  claimed in (\ref{eq:claim-u-sigma}).\\
Next we observe that (\ref{eq:claim-u-sigma}) is impossible if $k \ge 2$, since then every unit vector $\sigma \in \R^k$ can be connected with $-\sigma$ by a continuous curve of unit vectores, whereas $u_{-{\sigma}}=- u_{{\sigma}}$.
 
So we conclude that $k=1$, and we write $u_{t}$ in place of $u_{{\sigma}}$ for ${\sigma} = 1$.  Replacing $u$ by $-u$ if necessary,  we may
assume by (\ref{eq:claim-u-sigma}) that $u_t > 0$ in $\Omega$. For $x \in \cN$ we now define $S_x:= \{t \in \R\: :\: (t,x) \in \Omega\} \subset \R$, 
and we consider a nonempty connected component $S \subset S_x$. Then the function $t \mapsto u(t,x)$ is strictly increasing in $S$. 
Moreover, if $t_1= \inf S$ and $t_2 = \sup S$, then $(t_1,x), (t_2,x) \in \partial \Omega$ and $u(t_1,x)<u(t_2,x)$, which implies that 
$u(t_1,x)=-\sqrt{\lambda_0}$ and $u(t_2,x)=\sqrt{\lambda_0}$.  We thus have the following property:
\begin{equation}
  \label{eq:claim-connected-component}
  \begin{aligned}
&\text{If $x \in \cN$ and $S$ is a nonempty connected component of $S_x$, then}\\
&\text{$t \mapsto u(t,x)$ is an increasing 
homeomorphism from $S$ to $(-\sqrt{\lambda_0},\sqrt{\lambda_0})$.}  
 \end{aligned}
\end{equation}
Next we claim the following:
\begin{equation}
  \label{eq:claim-projection}
\begin{aligned}
&\text{For every $x \in \cN$ there exists precisely one}\\
&\text{$\tau= \tau(x) \in \R$ with $(\tau(x),x) \in \Omega$ and $u(\tau(x),x)=0$.}  
\end{aligned}
\end{equation}
Indeed, let $\cN_0 \subset \cN$ denote the set of all $x \in \cN$ such that $(t,x) \in \Omega$ and $u(t,x)=0$ for some $t \in \R$. Then 
$\cN_0$ is open and nonempty, since $\Omega$ is open and, by (\ref{eq:claim-connected-component}), for every $(t,x) \in \Omega$ there exists $\tilde t \in \R$ such that $(\tilde t,x) \in \Omega$ and $u(\tilde t,x)=0$. Moreover, $\cN_0$ is closed in $\cN$. Indeed, let $(x_n)_n$ be a sequence in $\cN_0$ with $x_n \to x \in \cN$ as $n \to \infty$, and let $t_n \in \R$, $n \in \N$ be such that $(t_n,x_n) \in \Omega$ and $u(t_n,x_n)=0$.  Since $\Omega$ is bounded, we may pass to a subsequence such that $t_n \to t$ as $n \to \infty$. We then have $(t,x) \in \overline \Omega$ and $u(t,x)=0$. Hence $(t,x) \not \in \partial \Omega$ since $u^2 \equiv \lambda_0>0$ on $\partial \Omega$. Consequently, $(t,x) \in \Omega$ and therefore $x \in \cN_0$. In sum, it follows that $\cN_0 = \cN$ since $\cN$ is connected, and thus for 
every $x \in \cN$ there exists at least one $t \in \R$ with $(t,x) \in \Omega$ and $u(t,x)=0$. Combining this with the fact that $u$ does not vanish on $\partial \Omega$, we see that the functions
$$
t_\pm: \cN \to \R,\qquad 
\left\{
  \begin{aligned}
t_-(x)&:= \min\{t \in \R\::\: \text{$(t,x) \in \Omega$ and $u(t,x)=0$}\}\\
t_+(x)&= \max\{t \in \R\::\: \text{$(t,x) \in \Omega$ and $u(t,x)=0$}\}
  \end{aligned}
\right.
$$
are well defined, and that $(t_\pm(x),x) \in \Omega$ for every $x \in \cN$.  Moreover, since $u_t>0$ in $\Omega$, it follows from the implicit function theorem that these functions are continuous. As a consequence, the open sets 
$$
\Omega_-:= \{(t,x) \in \Omega\::\: t<t_-(x)\}, \qquad \Omega_+:= \{(t,x) \in \Omega\::\: t>t_+(x)\}
$$
and $\Omega_0:= \{(t,x) \in \Omega\::\: t_-(x)<t<t_+(x)\}$ are disjoint, and $\Omega_\pm \not = \varnothing$ since $(t_\pm(x),x) \in \Omega$ for every $x \in \cN$. Since $u$ has precisely two nodal domains by the Courant nodal domain theorem, it follows that $\Omega_0= \varnothing$ and therefore $t_-(x)=t_+(x)$ for all $x \in \cN$. Thus (\ref{eq:claim-projection}) is true, and the function $\cN \to \R,\:x \mapsto \tau(x)$ is continuous. Moreover,  as a consequence of (\ref{eq:claim-connected-component}) we have, for $(t,x) \in \Omega$,
\begin{equation}
  \label{eq:conclusion-two-claims}
\text{$u(t,x)<0\,\,$ iff $\,\, t<\tau(x)\qquad$ and $\qquad u(t,x)>0\,\,$ iff $\,\,t> \tau(x)$.}
\end{equation}
Next we consider the disjoint sets 
$$
\M_+= \{(t,x) \in \M \setminus \Omega \::\: t > \tau (x)\}\quad \text{and}\quad 
\M_-= \{(t,x) \in \M \setminus \Omega \::\: t < \tau(x)\},
$$
and we set $\Gamma_\pm := \partial \M_\pm$. It then follows that $\M \setminus \Omega= \M_+ \cup \M_-$ and $\partial \Omega = \Gamma_+ \cup \Gamma_-$, whereas 
$u \equiv \sqrt{\lambda_0}$ on $\Gamma_+$ and $u \equiv -\sqrt{\lambda_0}$ on $\Gamma_-$ by (\ref{eq:conclusion-two-claims}).
Since $|\n u| =0$ on $\de\O$, we may therefore extend $u$ to a $C^1$-function on $\M$ by setting $u\equiv \sqrt{\lambda_0}$ on $ \M_+$ and $u \equiv -\sqrt{\lambda_0}$ on $
\M_-$. Next, we fix $r>0$ such that $\mu_2(\Omega)= \frac{\pi^2}{4r^2}$, and 
we consider the functions
$$
v_s: \M \to \R,\qquad u_s(t,x)= \left \{
  \begin{aligned}
&-\sqrt{\lambda_0}  &&\qquad t \le -r- s,\\
&\sqrt{\lambda_0} \sin( \frac{\pi}{2r} (t+s)),&&\qquad -r- s <t < r-s,\\
&\sqrt{\lambda_0}  &&\qquad t \ge r- s.
   \end{aligned}
\right.
$$
for $s \in \R$. Moreover, we set 
$$
s_+:= \min \{t \::\: (t,x) \in \Gamma_+\} \qquad \text{and}\qquad  s_-:= \min \{t \::\: (t,x) \in \Gamma_-\}= \inf \{t\::\: (t,x) \in \Omega\} .
$$
If $s>0$ is chosen sufficiently large, we have $v_s \equiv \sqrt{\lambda_0}$ on $\Omega \cup \M_+$.  Hence we may consider 
$$
s_0:= \inf \{s  \ge r- s_+ \::\: \text{$v_s \ge u$ on $\M$}\}.
$$
Writing $\tilde v$ instead of $v_{s_0}$, we see that $\tilde v \ge u$ on $\M$ by continuity. Since $\tilde v(t,\cdot) \equiv - \sqrt{\lambda_0}$ for $t \le -r -s_0$ and $u >-\sqrt{\lambda_0}$ in $\Omega$, we infer that 
\begin{equation}
  \label{eq:s_0-inequa}
s_- \ge -r-s_0 .
 \end{equation}
 Moreover, setting $\tilde \Omega:= \{(t,x) \in \M\::\: -r -s_0 < t < r-s_0\}$, we have that 
\begin{equation}
  \label{eq:equation-difference}
-\Delta_g (\tilde v -u)= \mu_2(\Omega) (\tilde v- u)\qquad \text{in $\tilde
  \Omega \cap \Omega$}.  
\end{equation}
We distinguish the following cases.\\
\textbf{\em Case 1:} There is a point $(t_0,x_0) \in \Omega$ such that $u(t_0,x_0)=
\tilde v(t_0, x_0)$. In this case, we have $\tilde v(t_0, x_0) = u(t_0, x_0) \in
(-\sqrt{\lambda_0},\sqrt{\lambda_0})$, so that $(t_0,x_0) \in \tilde \Omega$. 
By (\ref{eq:equation-difference}) and the strong maximum principle, we then conclude that 
$\tilde v \equiv u$ in the
connected component $\cZ$ of $\tilde \Omega \cap \Omega$ containing $(t_0,x_0)$. 
Since $\partial \cZ \subset \partial \tilde \Omega \cup \partial \Omega$,
we infer that $\tilde v^2= u^2 = \lambda_0$ on $\partial
\cZ$ by continuity. We claim that $\Omega  \subset \cZ$. Indeed, let $(t,x) \in
\Omega$, and let $\gamma$ be a curve joining $(t,x)$ and $(t,x_0)$ within
$\Omega$. Then $u^2<\lambda_0$ along $\gamma$, and therefore $\gamma$ does
not intersect $\partial \cZ$. Hence $(t,x) \in \cZ$. We conclude that
$\Omega \subset \cZ$. Using that $\tilde v^2<\lambda_0$ in $\tilde \Omega$, we
similarly conclude that $\tilde \Omega \subset \cZ$. Consequently we have 
$\tilde \Omega \cup \Omega \subset \cZ \subset \tilde \Omega \cap \Omega$ and  
hence $\tilde \Omega= \Omega$, which means that 
$\Omega=\Omega_r$ and $\tilde v = u$ in $\Omega_r$ after translation in the $t$-variable.\\ 
\textbf{\em Case 2:} $\tilde v > u$ in $\Omega$ and $s_0=r -s_+$. 
In this case there exists $x_+ \in \cN$ such that $(s_+,x_+) \in \Gamma_+ \cap \partial \tilde \Omega$. Moreover, the outer normal of $\Omega$ at $(s_+,x_+)$ and the outer normal of $\tilde \Omega$ at $(s_+,x_+)$ are both given by $\nu=(1,0) \in \R \times T_{x_+} \cN$. Consequently, by (\ref{eq:equation-difference}) and since 
$u<\tilde v$ in $\Omega$ and $u(s_+,x_+)= \lambda_0= \tilde v(s_+,x_+)$, the Hopf boundary lemma implies that $\partial_{\eta} (\tilde v- u)(s_+,x_+)<0$. This however is
impossible since $\nabla u(s_+,x_+)=\nabla \tilde v(s_+,x_+)=0$.\\
\textbf{\em Case 3:} $\tilde v > u$ in $\Omega$ and $s_0>r - s_+$. 
 In this case we claim that 
\begin{equation}
  \label{eq:hopf-lemma-claim}
s_- = -r-s_0.
\end{equation}
Indeed, if -- recalling (\ref{eq:s_0-inequa}) -- we suppose by contradiction that $s_->-r-s_0$, then 
$\tilde v > u$ in $\Omega \cup \Gamma_-$, and this easily easily implies that $v_{s_0-\eps} \ge u$ on $\M$ for $\eps>0$ sufficiently small, contradicting the definition of $s_0$.  Hence (\ref{eq:hopf-lemma-claim}) is true. Arguing similarly as in Case 2, we now consider $x_+ \in \cN$ such that $(s_-,x_-) \in \Gamma_- \cap \partial \tilde \Omega$. In this case, the outer normal of $\Omega$ at $(s_-,x_-)$ and the outer normal of $\tilde \Omega$ at $(s_-,x_-)$ are both given by $\nu=(-1,0) \in \R \times T_{x_-} \cN$.
Noting that $u(s_-,x_-)= -\sqrt{\lambda_0}= \tilde v(s_-,x_-)$, we arrive at a contradiction via the Hopf boundary lemma as in Case 2. 

Hence Case 1 must occur, and in this case we already concluded that 
$\Omega=\Omega_r$ and $\tilde v = u$ in $\Omega_r$ after translation in the $t$-variable. From the definition of $\tilde v$ we then deduce that $\mu_2(\Omega)= \mu^{r}$, which, by the separation of variables argument given in the proof of Theorem~\ref{sec:weinberger-modified}, implies that $\mu^{r} \le \mu_2(\cN)$ and therefore $r \ge \Bigl(\frac{\mu(B^k)}{\mu(\cN)}\Bigr)^{\frac{1}{2}}$. The proof is thus finished.
\end{proof}

\end{document}